\documentclass[12pt,a4paper,reqno]{amsart}

\usepackage[dvipsname,usenames]{xcolor}
\usepackage[utf8]{inputenc}
\usepackage[T1]{fontenc}
\usepackage[american]{babel}
\usepackage{amsmath,amssymb,amsthm,amscd}
\usepackage{mathrsfs}
\usepackage[active]{srcltx}
\usepackage{pgf,tikz}
\usepackage{mathrsfs}
\usepackage{enumerate}
\usetikzlibrary{arrows}
\usepackage{paralist}
\usepackage{dsfont}
\usepackage{bbm}
\usepackage{xfrac}

\usepackage{textcomp}

\let\OLDthebibliography\thebibliography
\renewcommand\thebibliography[1]{
  \OLDthebibliography{#1}
  \setlength{\parskip}{3pt}
  \setlength{\itemsep}{2.5pt plus 0.5ex}
}

\usepackage[colorlinks,linkcolor={black},citecolor={black},urlcolor={black}]{hyperref}
\usepackage[margin=2.6cm,bmargin=3.2cm,tmargin=3.2cm]{geometry}

\newcommand{\hm}{\widehat{m}}
\newcommand{\hJ}{\widehat{J}}
\newcommand{\hA}{\widehat{A}}

\newcommand{\cP}{\mathscr{P}}

\newcommand{\cX}{\mathcal{X}}

\newcommand{\cA}{\mathcal{A}}
\newcommand{\cU}{\mathcal{U}}
\newcommand{\cM}{\mathcal{M}}
\newcommand{\sM}{\mathscr{M}}

\newcommand{\cW}{\mathcal{W}}
\newcommand{\sd}{\mathsf{d}}

\newcommand{\bA}{\mathbb{A}}
\newcommand{\bT}{\mathbb{T}}

\newcommand{\Haus}{\mathscr{H}}

\newcommand{\one}{{{\bf 1}}}

\newcommand{\Lip}{\mathrm{Lip}}
\newcommand{\Leb}{\mathscr{L}}

\newcommand{\bCE}{\mathbb{CE}}
\newcommand{\CE}{\mathsf{CE}}

\newcommand{\R}{\mathbb{R}}
\newcommand{\Z}{\mathbb{Z}}
\newcommand{\N}{\mathbb{N}}
\newcommand{\dd}{\, \mathrm{d}}
\newcommand{\ddd}{\mathrm{d}}

\newcommand{\eps}{\varepsilon}

\newcommand{\tand}{\quad\text{ and }\quad}

\newcommand{\ddt}{\frac{\mathrm{d}}{\mathrm{d}t}}

\newcommand{\suchthat}{\ensuremath{\ : \ }} % such that inside, for example, the sets definition

\theoremstyle{plain}
\newtheorem{theorem}{Theorem}[section]
\newtheorem{corollary}[theorem]{Corollary}
\newtheorem{lemma}[theorem]{Lemma}
\newtheorem{proposition}[theorem]{Proposition}

\newtheorem{definition}[theorem]{Definition}

\theoremstyle{remark}
\newtheorem{remark}[theorem]{Remark}
\newtheorem{example}[theorem]{Example}

\newtheorem*{claim*}{Claim}
\newtheorem*{remark*}{Remark}
\newtheorem*{example*}{Example}
\newtheorem*{notation*}{Notation}
\numberwithin{equation}{section}

\definecolor{jan}{rgb}{0.0,0.3,0.8}
\definecolor{mat}{rgb}{0.0,0.5,0.3}

\newcommand{\cT}{\mathcal{T}}
\newcommand{\cS}{\mathcal{S}}
\newcommand{\cV}{\mathscr{V}}

\newcommand{\emm}{m}
\newcommand{\dive}{\partial_x}
\newcommand{\supp}{\operatorname{supp}}
\newcommand{\diam}{\operatorname{diam}}

\newcommand{\bW}{\mathbb{W}_2}
\newcommand{\bM}{\mathbb{M}}

\title{Homogenisation of one-dimensional discrete optimal transport}

\author{Peter Gladbach}
\address{Mathematisches Institut, Universit\"at Leipzig, Augustusplatz 5, 04103 Leipzig}
\email{gladbach@math.uni-leipzig.de}

\author{Eva Kopfer}
\address{Institut f\"ur angewandte Mathematik, Universit\"at Bonn, Endenicher Allee 60, 53115 Bonn, Germany}
\email{eva.kopfer@iam.uni-bonn.de}

\author{Jan Maas}
\author{Lorenzo Portinale}
\address{Institute of Science and Technology Austria (IST Austria),
Am Campus 1, 3400 Klosterneuburg, Austria}
\email{jan.maas@ist.ac.at}
\email{lorenzo.portinale@ist.ac.at}

\begin{document}

\begin{abstract}
This paper deals with dynamical optimal transport metrics defined by spatial discretisation of the Benamou--Benamou formula for the Kantorovich metric $\bW$.
Such metrics appear naturally in discretisations of $\bW$-gradient flow formulations for dissipative PDE. 
However, it has recently been shown that these metrics do not in general converge to $\bW$, unless strong geometric constraints are imposed on the discrete mesh. 
In this paper we prove that, in a $1$-dimensional periodic setting, discrete transport metrics converge to a limiting transport metric with a non-trivial effective mobility. This mobility depends sensitively on the geometry of the mesh and on the non-local mobility at the discrete level.
Our result quantifies to what extent discrete transport can make use of microstructure in the mesh to reduce the cost of transport.
\end{abstract}
\maketitle

%\tableofcontents

\section{Introduction}
\label{sec:intro}

In the past decades there has been intense research activity in the area of  optimal transport, cf.~the monographs \cite{Villani:2003,Villani:2008,Santambrogio:2015,Peyre-Cuturi:2019} for an overview of the subject.
In continuous settings, a key result in the field is the \emph{Benamou--Brenier formula} \cite{Benamou-Brenier:2000}, which expresses the equivalence of static and dynamical formulations of the optimal transport problem. 
In discrete settings, the equivalence between static and dynamical optimal transport breaks down, and it turns out that the dynamical formulation (introduced in \cite{Maas:2011,Mielke:2011}) is essential in applications to evolution equations, discrete Ricci curvature, and functional inequalities, see, e.g., \cite{Chow-Huang-Li:2012,Erbar-Maas:2012,Mielke:2013,Erbar-Maas:2014, Erbar-Maas-Tetali:2015, Fathi-Maas:2016,Erbar-Henderson-Menz:2017,Erbar-Fathi:2018}.

However, the limit passage from discrete dynamical transport to continuous optimal transport turns out to be nontrivial.
In fact, it has been shown in \cite{Gladbach-Kopfer-Maas:2018} that seemingly natural  discretisations of the Benamou--Brenier formula do \emph{not} necessarily converge to the Kantorovich distance $\bW$, even in one-dimensional settings.
The main result in \cite{Gladbach-Kopfer-Maas:2018} asserts that, for a sequence of meshes on a bounded convex domain in $\R^d$,  an isotropy condition on the meshes is required to obtain the convergence of the discrete dynamical transport distances to $\bW$.

It remained an open question to identify the limiting behaviour of the discrete metrics in situations where the isotropy condition fails to hold. 
The aim of the current paper is to answer this question in the one-dimensional periodic setting.

\medskip

We start by informally introducing the main objects of study in this paper and present the main result. For more formal definitions we refer to Section \ref{sec:prel} below.

\subsection*{Continuous optimal transport}
Let $\cP(\cS)$ (resp. $\sM(\cS)$) denote the set of Borel probability measures (resp. signed measures) on a Polish space $(\cS, \sd)$. 
We will work on the one-dimensional torus $\cS^1 = \R / \Z$ and use the convention that arithmetic operations are understood \emph{modulo $1$}.

The \emph{Kantorovich metric} $\bW$ (also known as \emph{Wasserstein metric}) on $\cP(\cS)$ is defined by 
\begin{align}\label{eq:Wp}
	\bW^2(\mu_0, \mu_1)
	= \inf_{\gamma \in \Gamma(\mu_0, \mu_1)}
	 \bigg\{\int_{\cS \times \cS} \sd^2(x,y) \dd \gamma(x,y)  \bigg\}
\end{align}
for $\mu_0, \mu_1 \in \cP(\cS)$. Here, $\Gamma(\mu_0, \mu_1)$ denotes the set of probability measures on $\cS \times \cS$ with marginals $\mu_0$ and $\mu_1$ respectively. 
For $\mu_0, \mu_1 \in \cP(\cS^1)$ the \emph{Benamou--Brenier formula} yields the equivalent dynamical formulation
\begin{align}\label{eq:BB}
	\bW^2(\mu_0, \mu_1) = 
	 \inf_{\mu, v} \bigg\{ 
		\int_0^1 \int_{\cS^1}
		 \frac{|j|^2}{\mu}  \suchthat \partial_t \mu + \partial_x j = 0 \bigg\} \ ,
\end{align} 
where the infimum runs over all curves $\mu : [0,1] \to \cP(\cS^1)$ connecting $\mu_0$ and $\mu_1$, and all vector fields $j : [0,1] \times \cS^1 \to \R$ satisfying the stated continuity equation.
Here, $\int_0^1 \int_{\cS^1} \frac{|j|^2}{\mu}$ is to be understood as $\int_0^1 \int_{\cS^1} |v_t(x)|^2 \dd \mu_t(x) \dd t$ if $j \ll v$ with $\frac{\ddd j}{\ddd \mu} = v$, and $+\infty$ otherwise.

\subsection*{Discrete dynamical optimal transport}

Let $\cX$ be a finite set endowed with a reference probability measure $\pi \in \cP(\cX)$. Let $R : \cX \times \cX \to \R_+$ denote the transition rates of an irreducible continuous time Markov chain on $\cX$. We assume that the \emph{detailed balance condition} holds, i.e., $\pi(x) R(x,y) = \pi(y) R(y,x)$ for all $x, y \in \cX$.

Let $\{\theta_{xy}\}_{x,y \in \cX}$ be a collection of admissible means, i.e., each $\theta_{xy} : \R_+ \times \R_+ \to \R_+$ is concave, $1$-homogeneous, and satisfies $\theta(1,1) = 1$. 
We assume that $\theta_{xy}(a,b) = \theta_{yx}(b,a)$ for any $a, b \geq 0$.

The \emph{discrete dynamical transport metric} associated to $(\cX, R, \pi)$ is defined by
\begin{align*}
	\cW^2(m_0, m_1) =
	 \inf_{m,J}\Bigg\{
		\frac12 \int_0^1 \sum_{x,y \in \cX}
		 \frac{J_t^2(x, y)}
		 {\theta_{xy}\big( m_t(x) R(x,y) , m_t(y) R(y,x) \big)}
		  \dd t \Bigg\} \ .
\end{align*}
Here the infimum runs over all curves $m : [0,1] \to \cP(\cX)$ connecting $m_0$ and $m_1$, and all discrete vector fields $J : [0,1] \to \cV(\cX)$ satisfying the discrete continuity equation
\begin{align*}
      \ddt m_t(x) + \sum_{y \in \cX} J_t(x, y)  = 0 
      \quad \quad \text{ for all } x \in \cX \ ,
\end{align*}
where $\cV(\cX)$ denotes the set of all anti-symmetric functions $V : \cX \times \cX \to \R$.
The definition of $\cW$ is a direct analogue of \eqref{eq:BB} with one additional feature: between any pair of points $x$ and $y$ an admissible mean $\theta_{xy}$ needs to be chosen to describe the mobility.

\subsection*{Discrete optimal transport on $1$-dimensional meshes}

In this paper we consider discrete transport metrics induced by a finite volume discretisation of $\cS^1$.

Fix $0 = r_0 < \ldots < r_1 < \ldots < r_K = 1$ for some $K \geq 1$. 
We write $\pi_k := r_{k+1} - r_k$ and $A_k := [r_k, r_{k+1})$, so that $\cT := \{ A_k \}_{k=0}^{K-1}$ is a partition of $\cS^1$ into disjoint half-open intervals.
We also consider a sequence of points $\{ z_k \}_{k=0}^{K-1}$ such that each $z_k$ lies in the interior of $A_k$.
The distance between $z_k$ and $z_{k'}$ in $\cS^1$ will be denoted by $d_{kk'}$. Here and below we will often perform calculations \emph{modulo $K$}.

\begin{figure}[ht]
\centering
\begin{tikzpicture}[xscale=0.8]
\draw [thick] (0,0)  -- (12,0);
\draw (0,-.1)  -- (0, .1) ; \node [above] at (-0.25, .2) {$0=r_0$};
\draw (3,-.1) -- (3, .1); \node [above] at (3, .2) {$r_1$};
\draw (5.5,-.1) -- (5.5, .1); \node [above] at (5.5, .2) {$r_2$};
\draw[thick, loosely dotted] (7, 0.4) -- (8, 0.4) ; 
\draw (10,-.1) -- (10, .1); \node [above] at (10, .2) {$r_{K-1}$};
\draw (12, -.1) -- (12, .1); \node [above] at (12.35, .2) {$r_K = 1$};
\draw[red, fill] (2,0)  circle [radius=0.050] node [below, red] at (2,0) {$z_0$};
\draw[red, fill] (3.5,0) circle [radius=0.050]; \node [below, red] at (3.5,0) {$z_1$};
\draw[red, fill] (7.5,0) circle [radius=0.050]; \node [below, red] at (7.5,0) {$z_{K-2}$};
\draw[red, fill] (10.5,0) circle [radius=0.050]; \node [below, red] at (10.5,0) {$z_{K-1}$};

\draw[thick,blue][<->] (0.1,1) -- node[above] {$\pi_0$}  (2.9,1); 
\draw[thick,blue][<->] (3.1,1) -- node[above] {$\pi_1$}  (5.4,1); 
\draw[thick, loosely dotted] (7, 1) -- (8, 1) ;
\draw[thick,blue][<->] (10.1,1) -- node[above] {$\pi_{K-1}$}  (11.9,1);
\end{tikzpicture}
 \caption{The mesh $\cT$ on $\cS^1$.}
\label{fig:mesh-1-intro}
\end{figure}
We endow the discrete state space $\cT$ with the natural reference measure $\pi \in \cP(\cT)$ given by $\pi(A_k) = \pi_k$. 
The main object of study in this paper is the transport metric $\cW_\cT$ on $\cP(\cT)$ induced by the Markov transition rates on $\cT$ given by
\begin{align*}
	R(A_k, A_{k'}) :=	R_{k k'} := \frac{1}{\pi_k d_{kk'}} 
\end{align*}
if $|k - k'| = 1$, and $R_{kk'} = 0$ otherwise.
Then we have the detailed balance condition $\pi_k 	R_{k k'} = \pi_{k'} R_{k' k}$.
The rates are chosen to ensure that solutions to the discrete diffusion equation (i.e., the Kolmogorov forward equation associated to the Markov chain given by $R$) converge to solutions of the diffusion equation $\partial_t \mu = \partial_x^2 \mu$ in the limit of vanishing mesh size \cite{Eymard-Gallouet-Herbin:2000}. A gradient flow approach in one dimension can be found in \cite{Disser-Liero:2015}.

\subsection*{The periodic setting}

For any mesh $\cT$ as above and $N \geq 1$ one can construct an inhomogeneous periodic mesh $\cT_N$ with $NK$ cells $A_{n;k}$ by concatenating $N$ rescaled copies of $\cT$.

\begin{figure}[ht]
\begin{center}
\begin{tikzpicture}[xscale=0.4]
\draw [thick] (0,0)  -- (36,0); \node[above] at (0,.3) {0}; \node[above] at (36,.3) {1};
\draw[line width= 0.75mm] (0,-.3)  -- (0, .3) ; 
\draw (3,-.1) -- (3, .1); 
\draw (5.5,-.1) -- (5.5, .1); 
\draw (10,-.1) -- (10, .1);
\draw[line width= 0.75mm] (12, -.3) -- (12, .3); 
\draw[red, fill] (1.7,0)  circle [radius=0.050]; 
\draw[red, fill] (3.7,0) circle [radius=0.050]; 
\draw[red, fill] (7.5,0) circle [radius=0.050];
\draw[red, fill] (10.5,0) circle [radius=0.050]; 
\draw (15,-.1) -- (15, .1); 
\draw (18.5,-.1) -- (18.5, .1); 
\draw (22,-.1) -- (22, .1); 
\draw[line width= 0.75mm] (24, -.3) -- (24, .3); 
\draw[thick,blue][<->] (24.1,1) -- node[above] {$\frac{1}{N}$} (35.9,1) ;
\draw[red, fill] (13.7,0)  circle [radius=0.050];
\draw[red, fill] (15.7,0) circle [radius=0.050]; 
\draw[red, fill] (19.5,0) circle [radius=0.050]; 
\draw[red, fill] (22.5,0) circle [radius=0.050];
\draw (27,-.1) -- (27, .1); 
\draw (29.5,-.1) -- (29.5, .1); 
\draw (34,-.1) -- (34, .1); 
\draw[line width= 0.75mm] (36, -.3) -- (36, .3); 
\draw[red, fill] (25.7,0)  circle [radius=0.050] ;
\draw[red, fill] (27.7,0) circle [radius=0.050]; 
\draw[red, fill] (31.5,0) circle [radius=0.050]; 
\draw[red, fill] (34.5,0) circle [radius=0.050]; 

\draw[thick,blue][<->] (0.1,1) -- node[above] {$\frac{\pi_0}{N}$}  (2.9,1); 
\draw[thick,blue][<->] (3.1,1) -- node[above] {$\frac{\pi_1}{N}$}  (5.4,1); 
\draw[thick, loosely dotted] (7, 1) -- (9, 1) ;
\draw[thick,blue][<->] (10.1,1) -- node[above] {$\frac{\pi_{K-1}}{N}$}  (11.9,1);
\end{tikzpicture}
\end{center}
    \caption{The mesh $\cT_N$ on $\cS^1$.}
\label{fig:mesh-2-intro}
\end{figure}
We then consider the transport metric $\cW_N := \cW_{\cT_N}$ on $\cP(\cT_N)$ as defined above. 
Explicitly, we have
\begin{align*}
	\cW_N^2(m_0, m_1) = 
	 \inf_{m,J}  \left\{ \frac{1}{N}
		 \int_0^1 
		\sum_{n=0}^{N-1} \sum_{k=0}^{K-1} d_{k,k+1}
				\frac{ J_t^2(n;k,k+1)}{
	\theta_{k,k+1}\Big( \frac{Nm_t(n;k)}{\pi_{k}}, 
					    \frac{Nm_t(n;k+1)}{\pi_{k+1}} \Big)} 
		  \dd t \right\}
		   \ ,
\end{align*}
where the infimum runs over all curves $m : [0,1] \to \cP(\cT_N)$ and $J : [0,1] \to \cV(\cT_N)$ satisfying the discrete continuity equation
\begin{align*}
      \ddt m_t(n;k) + J_t(n;k,k+1) - J_t(n;k-1, k) = 0
\end{align*}
for all $n = 0, \ldots, N-1$ and $k = 0, \ldots, K-1$.
Here we use the shorthand notation $m(n;k) = m(A_{n;k})$ and $J(n;k,k+1) = J(A_{n;k}, A_{n;k+1})$. Moreover, we use the convention that $m(n;K)=m(n+1;0)$ and $J(n;K-1,K)=J(n+1;-1,0)$. The main goal of this paper is to analyse the limiting behaviour of $\cW_N$ as $N \to \infty$.

\subsection*{The discrete-to-continuous limit}

The first convergence result for discrete dynamical transport metrics (in the sense of Gromov--Hausdorff) was obtained in \cite{Gigli-Maas:2013}. 
There it is shown that the discrete transport metric associated to the cubic mesh on the $d$-dimensional torus converges to $\bW$ in the limit of vanishing mesh size.

The limiting behaviour of discrete dynamical transport metrics on more general meshes turns out to be a delicate issue.
In fact, it follows from the multi-dimensional results in \cite{Gladbach-Kopfer-Maas:2018} that the discrete transport metrics $\cW_N$ converge to $\bW$ if and only if the means $\theta_{kk'}$ are carefully chosen to satisfy an appropriate ``balance condition'' that reflects the geometry of the mesh $\cT$. 
In our one-dimensional periodic setting, these results imply that $\cW_N$ converges to $\bW$ if and only if there exist constants $\lambda_{k,k+1}, s \in (0,1)$ such that the following conditions hold for $k = 0, \ldots, K-1$: 
\begin{equation}\begin{aligned}\label{eq:compatibility}
	 r_{k+1} & = \lambda_{k,k+1} z_{k+1}  +  ( 1-\lambda_{k,k+1}) z_k + s \ , \\
	\theta_{k,k+1}(a,b) & \leq  \lambda_{k,k+1} a \ \ \  \; + ( 1-\lambda_{k,k+1}) b
		 \qquad \text{ for any } a, b \geq 0 \ .
\end{aligned}\end{equation}
Thus, to fulfill this condition, the asymmetry of the means $\theta_{k,k+1}$ should reflect the relative location of the points $z_k$, $r_{k+1}$, and $z_{k+1}$.
We refer to Section \ref{sec:c-star} below for a full discussion.

The main contribution of the current paper is the identification of the limiting behaviour of $\cW_N$ in the general one-dimensional periodic setting, without assuming \eqref{eq:compatibility}.
To state the result, we introduce the canonical projection operator
 $P_\cT : \cP(\cS^1) \to \cP(\cT)$ defined by
\begin{align} \label{eq: def of projection map}
	(P_\cT \mu) (\{A\}) = \mu(A)
\end{align}
for $\mu \in \cP(\cS^1)$ and $A \in \cT$. 
For brevity we write $P_N := P_{\cT_N}$.

The following homogenisation result asserts that $\cW_N$ converges to a Kantorovich metric with an effective mobility determined by the geometry of the mesh and by the choice of the means $\theta_{k,k+1}$. 

\begin{theorem}[Main result]\label{thm:main}
Fix a mesh $\cT$ on $\cS^1$, and consider the induced periodic meshes $\cT_N$ for $N \geq 1$.
For any $\mu_0, \mu_1 \in \cP(\cS^1)$, we have
\[
\lim_{N \rightarrow \infty}
		\cW_N(P_N \mu_0,P_N \mu_1)
			 = \sqrt{c^\star(\theta, \cT)} \bW(\mu_0,\mu_1) \ ,
\]
where
\begin{equation} \label{eq: hom-constant for K-periodic case}
c^\star(\theta, \cT) := \inf \Bigg \{  \sum_{k=0}^{K-1} \frac{d_{k,k+1}}{\theta_{k,k+1}\Big( \frac{m_k}{\pi_k} , \frac{m_{k+1}}{\pi_{k+1}} \Big)} \suchthat 
m \in \cP(\cT)
\Bigg \} \ .
\end{equation}
Moreover, as $N \to \infty$ we have Gromov--Hausdorff convergence of metric spaces:
\begin{align*}
	( \cP(\cT_N), \cW_N )
	 	\to
		\big(\cP(\cS^1), \sqrt{c^\star(\theta, \cT)} \bW \big) \ .
\end{align*}
\end{theorem}

\begin{remark}[Upper bound and isotropic case]\label{rem:isotropic}
We show in Section \ref{sec:c-star} that $c^*(\theta,\cT) \leq 1$. Moreover, if the compatibility conditions \eqref{eq:compatibility} are satisfied, it follows that $c^\star(\theta, \cT) = 1$, and we recover the result of \cite{Gladbach-Kopfer-Maas:2018}.
\end{remark}

\begin{remark}[Convergence of gradient flows]\label{rem:EDP-convergence}
We stress that the limiting behaviour at the level of the transport metrics is in stark contrast with the convergence results of the level of the gradient flow equation.
Indeed, consider the discrete transport metric $\cW_N$ in the case where each $\theta_{k,k+1}$ is equal to the \emph{logarithmic mean} $\theta_{\rm log}(a,b) = \int_0^1 a^{1-s} b^s \dd s$.
Then the discrete diffusion equation is the gradient flow equation in $(\cP(\cT_N), \cW_N)$ for the relative entropy with respect to the natural reference measure $\pi_N$; cf. \cite{Chow-Huang-Li:2012,Maas:2011,Mielke:2011}.
Similarly, the continuous diffusion equation is the gradient flow in $(\cP(\cS^1), \bW)$ for the relative entropy with respect to the Lebesgue measure on $\cS^1$  \cite{Jordan-Kinderlehrer-Otto:1998}.
Convergence of solutions of the discrete heat equation to solutions of the continuous heat equation is well known, see, e.g., \cite{Eymard-Gallouet-Herbin:2000}.
Nevertheless, our main result shows that the discrete transport metrics $\cW_N$ converge to a limiting metric that is different from $\bW$, unless the mesh is equidistant. 
For a systematic study of convergence of gradient flow structures we refer to \cite{Mielke:2016,Mielke:2016a,Dondl-Frenzel-Mielke:2018}; see also \cite{Al-Reda-Maury:2017} for a discussion in the context of finite volume discretisations. 
\end{remark}

\begin{remark}[Convergence on geometric graphs]\label{rem:point-clouds}
A convergence result for discrete transport distances on a large class of geometric graphs associated to point clouds on the $d$-dimensional torus   has been obtained in \cite{Garcia-Trillos:2017}.
This result applies in particular to iid points sampled from the uniform distribution on the torus.
As the results in that paper apply to sequences of graphs with increasing degree, they do not overlap with the results obtained here.
\end{remark}

\subsection*{Heuristics}

We briefly sketch a non-rigorous argument that makes Theorem \ref{thm:main} plausible. 
For this purpose we consider a smooth solution to the continuity equation $\partial_t \mu + \partial_x j = 0$, and fix $\alpha \in \cP(\cT)$. 
Suppressing the time variable, we define a discrete measure $m$ that assigns mass $m(k) := \alpha(k) \mu\big([\tfrac{n}{N}, \tfrac{n+1}{N})\big)$ to each cell $A_{n;k}$ in $\cT_N$. This ensures that each interval of the form $[\tfrac{n}{N}, \tfrac{n+1}{N})$ receives the same mass at the discrete and the continuous level, but within each such interval, the measure $\alpha$ introduces discrete density oscillations.

Let $J$ be the discrete momentum vector field that solves the continuity equation for $m$. If this vector field is sufficiently regular, we may estimate the discrete energy by
\begin{align*}
		& \sum_{n=0}^{N-1} \sum_{k=0}^{K-1} d_{k,k+1}
				\frac{ J_t^2(n;k,k+1)}{
	\theta_{k,k+1}\Big( \frac{Nm_t(n;k)}{\pi_{k}}, 
					    \frac{Nm_t(n;k+1)}{\pi_{k+1}} \Big)} 
\\ &\qquad\qquad   \approx 
	\frac{1}{N} \sum_{n=0}^{N-1} 
			\frac{J_t^2(n;0,1)}{\mu\big([\tfrac{n}{N}, \tfrac{n+1}{N})\big)}
			\sum_{k=0}^{K-1} 
				\frac{ d_{k,k+1} }
					{
	\theta_{k,k+1}\Big( \frac{\alpha(k)}{\pi_{k}}, 
					    \frac{\alpha(k+1)}{\pi_{k+1}} \Big)} 
\approx c^\star(\theta, \cT) \int \frac{|j|^2}{\mu} \ ,
\end{align*}
after minimisation over $\alpha \in \cP(\cT)$. 
We thus recover the continuous energy appearing in the Benamou--Brenier formula up to a multiplicative correction, which indeed suggests our main result.

A rigorous argument based on this heuristics clearly requires suitable spatial regularity results for $m$ and $J$.
Indeed, we will show in Section \ref{sec:lower} below that any discrete curve can be approximated by a curve of similar energy, which enjoys good Lipschitz bounds for $J$ as well as good Lipschitz bounds for $m$ up to oscillations within each cell.

\subsection*{Organisation of the paper}

In Section \ref{sec:prel} we collect the basic definitions and preliminary results that are used in this paper. 
In Section \ref{sec:simple} we give a simple approach to some of the main convergence results, which only applies in the special case where $\cT$ consists of exactly $2$ cells.
In Section \ref{sec:c-star} we analyse the formula \eqref{eq: hom-constant for K-periodic case} for the effective mobility $c^\star(\theta, \cT)$ and discuss its relation to the geometric conditions from \cite{Gladbach-Kopfer-Maas:2018}.

The bulk of the proof of the main result is contained in Sections \ref{sec:lower} and \ref{sec:upper}, which deal with the lower and upper bounds for $\cW_N$ respectively. The key results in these sections are Theorems \ref{thm:lowerbound_quantitative} and \ref{thm:upperbound_quantitative}.
In Section \ref{sec:GH} we finish the proof of the main result by proving the Gromov--Hausdorff convergence.

\section{Preliminaries}
\label{sec:prel} 

\subsection{Continuous optimal transport on \texorpdfstring{$\cS^1$}{S}}

For $\mu_0, \mu_1 \in \cP(\cS^1)$, let $\bCE(\mu_0,\mu_1)$ denote the set of all distributional solutions to the continuity equation
\begin{align}\label{eq:CE}
\partial_t \mu + \dive j = 0
\end{align}
with boundary conditions $\mu_t|_{t = 0} = \mu_0$ and $\mu_t|_{t = 1} = \mu_1$.
More precisely, this means that $(\mu_t)_t$ is a weakly continuous family of measures in $\cP(\cS^1)$ with the given boundary conditions, $(j_t)_t$ is a Borel family of measures in $\sM(\cS^1)$ satisfying 
$
	\int_0^1 |j_t|(\cS^1) \dd t < \infty
$,
and \eqref{eq:CE} holds in the sense that
\begin{align*}
	\int_0^1 \int_{\cS^1} 
		\partial_t \xi(t,x) \dd \mu_t(x) \dd t
		+ 
	\int_0^1 \int_{\cS^1} 
		\partial_x \xi(t,x) \dd j_t(x) \dd t 
		= 0 
\end{align*}
for any test function $\xi \in C_c^1(\cS^1 \times (0,1))$.
For $\mu \in \cP(\cS^1)$ and $j \in \sM(\cS^1)$ we set
\begin{align*}
	\bA(\mu, j) = 
	\int_{\cS^1} \bigg|\frac{\ddd j}{\ddd \mu}\bigg|^2 \dd \mu \ ,
\end{align*}
if $j \ll \mu$, and $\bA(\mu, j) = + \infty$ otherwise.
With this notation, the Benamou--Brenier formula \cite{Benamou-Brenier:2000} asserts that
\begin{equation} \label{eq:W2momenta}
\bW^2(\mu_0, \mu_1) =
	 \inf\bigg\{
		\int_0^1 \bA(\mu_t, j_t)  \dd t \suchthat ( \mu_t, j_t )_t \in \bCE(\mu_0,\mu_1) \bigg\} \ ,
\end{equation}
see, e.g., \cite[Lemma~8.1.3]{Ambrosio-Gigli-Savare:2008} for more details.

\subsection{Discrete optimal transport on one-dimensional meshes}

As in Section \ref{sec:intro}, we fix a mesh $\cT = \{ A_k \}_{k=0}^{K-1}$ on $\cS^1$, and use the notation $r_k$, $\pi_k$, $z_k$, $d_{kk'}$. 
The set $\cV(\cT)$ of discrete vector fields is naturally identified with the set of real-valued functions on $\{ (k,k+1) \}_{k=0}^{K-1}$.

\begin{figure}[ht]
\centering
\begin{tikzpicture}[xscale=0.8]
\draw [thick] (0,0)  -- (12,0);
\draw (0,-.1)  -- (0, .1) ; \node [above] at (-0.25, .2) {$0=r_0$};
\draw (3,-.1) -- (3, .1); \node [above] at (3, .2) {$r_1$};
\draw (5.5,-.1) -- (5.5, .1); \node [above] at (5.5, .2) {$r_2$};
\draw[thick, loosely dotted] (7, 0.4) -- (8, 0.4) ; 
\draw (10,-.1) -- (10, .1); \node [above] at (10, .2) {$r_{K-1}$};
\draw (12, -.1) -- (12, .1); \node [above] at (12.35, .2) {$r_K = 1$};
\draw[red, fill] (2,0)  circle [radius=0.050] node [below, red] at (2,0) {$z_0$};
\draw[red, fill] (3.5,0) circle [radius=0.050]; \node [below, red] at (3.5,0) {$z_1$};
\draw[red, fill] (7.5,0) circle [radius=0.050]; \node [below, red] at (7.5,0) {$z_{K-2}$};
\draw[red, fill] (10.5,0) circle [radius=0.050]; \node [below, red] at (10.5,0) {$z_{K-1}$};

\draw[thick,red][<->] (2,-.7) -- node[below] {$d_{01}$}  (3.5,-.7); 
\draw[thick,red][<->] (7.5,-.7) -- node[below] {$d_{K-2,K-1}$}  (10.5,-.7); 

\draw[thick,blue][<->] (0.1,1) -- node[above] {$\pi_0$}  (2.9,1); 
\draw[thick,blue][<->] (3.1,1) -- node[above] {$\pi_1$}  (5.4,1); 
\draw[thick, loosely dotted] (7, 1) -- (8, 1) ;
\draw[thick,blue][<->] (10.1,1) -- node[above] {$\pi_{K-1}$}  (11.9,1);
\end{tikzpicture}
 \caption{The mesh $\cT$ on $\cS^1$.}
\label{fig:mesh-1}
\end{figure}

\begin{definition}[Discrete continuity equation]\label{def:CE}
A pair $(m_t, J_t)_{t \in [0,1]}$ is said to satisfy the \emph{discrete continuity equation} if
  \begin{itemize}
  \item[(i)] $m : [0,1] \to \cP(\cT)$ is continuous;
  \item[(ii)] $J : [0,1] \to \cV(\cT)$ is locally integrable;
  \item[(iii)] the continuity equation holds in the sense of
    distributions:
    \begin{align}\label{eq:cont-eq}
      \ddt m_t(k) + J_t(k,k+1) - J_t(k-1, k) = 0 \quad \quad \text{ for all } k  = 0, \ldots, K-1 \;.
    \end{align}
  \end{itemize}
We write $\CE_\cT(m_0, m_1)$ to denote the collection of pairs $(m_t, J_t)_{t \in [0,1]}$ satisfying $m|_{t=0} = m_0$ and $m|_{t=1} = m_1$.
\end{definition}

\begin{definition}[Admissible mean]\label{def:mean}
An \emph{admissible mean} is a function $\theta : \R_+ \times \R_+ \to \R_+$ that is concave, $1$-homogeneous, and satisfies $\theta(1,1) = 1$.
\end{definition}

Note that we do \emph{not} impose that admissible means are symmetric. 
Let us briefly recall some properties of admissible means that will be used in the sequel.

\begin{lemma}[Properties of admissible means]
\label{lem:prop_admissible_means}
For any admissible mean $\theta : \R_+ \times \R_+ \to \R_+$ the following statements hold:
\begin{enumerate}[(i)]
	\item \label{item:bounds} For $a,b \geq 0$ we have $\min\{a,b\} \leq \theta(a,b) \leq \max \{a,b\}$.
	\item \label{item:jlsc} The map $\R_+ \times \R_+ \ni (a,b) \mapsto \frac{1}{\theta(a,b)} \in (0, +\infty]$ is jointly lower semicontinuous.
	\item \label{item:loclip} $\theta$ is locally Lipschitz on $(0,+\infty)^2$.
\end{enumerate}
\end{lemma}

\begin{proof}
For all $a,b,s,t \geq 0$ we obtain, using $1$-homogeneity and concavity,
\begin{align*}
	 \theta(a + s,b +t) 
	= 2\theta(\tfrac{a + s}{2},\tfrac{b + t}{2}) 
	\geq   \theta(a, b) 
	+	  \theta(s, t) 
	\geq   \theta(a, b),
\end{align*}
hence $\theta$ is non-increasing with respect to the first and the second variable.
Thus, if $a \leq b$, it follows that 
$
	a = 
		 \theta(a,a)
	\leq \theta(a,b) 
	\leq \theta(b,b) 
	  = b. 
$
Since the same argument applies if $a \geq b$, we obtain  \eqref{item:bounds}.

The claims in \eqref{item:jlsc} and \eqref{item:loclip} are easy consequences of the assumptions on $\theta$.
\end{proof}

\begin{definition}[Discrete dynamical transport distance]
\label{def:disc-trans}
	Let $\cT$ be a mesh on $\cS^1$, and $\{\theta_{k,k+1}\}_{k=0}^{K-1}$ be a family of admissible means.
	\begin{enumerate}
		\item
		The energy functional $\cA_\cT : \cP(\cT) \times \cV(\cT) \to \R \cup \{ + \infty\}$ is given by
		\begin{align*}
		\cA_\cT(m, J) = \sum_{k = 0}^{K-1}
			d_{k,k+1} f_{k,k+1}\bigg( \frac{m(k)}{\pi_k},  \frac{m(k+1)}{\pi_{k+1}},  J(k,k+1) \bigg)\ ,
		\end{align*}
		where $f_{k,k+1}(\rho, \tilde\rho, J) =
			F\big(\theta_{k,k+1}( \rho, \tilde\rho), J\big)$, and
		\[
		 	F(\rho, J) =
			 \left\{ \begin{array}{ll}
				\frac{J^2}{\rho}
			 & \text{if $\rho > 0$} \ ,\\
		0
			 & \text{if $\rho = 0$ and $J = 0$}\ ,\\
		+ \infty
			 & \text{otherwise} \ .\end{array} \right.
		\]
		\item
		The discrete dynamical transportation distance between $m_0, m_1 \in \cP(\cT)$ is given by
		\begin{equation*}
		\cW_\cT(m_0, m_1) =
		\inf \Bigg \{  \sqrt{\int_0^1  \cA_\cT(m_t, J_t) \dd t} \suchthat (m_t, J_t)_{t} \in \CE_\cT(m_0, m_1) \Bigg \} \ .
		\end{equation*}
	\end{enumerate}
\end{definition}

The infimum in the the previous definition is attained; cf.~\cite[Theorem 3.2]{Erbar-Maas:2012}. In the sequel we apply these definitions to the periodic meshes $\cT_N$ defined in Section \ref{sec:intro}. We will then simply write $\cA_N$ and $\cW_N$ as a shorthand for $\cA_{{\cT_N}}$ and $\cW_{\cT_N}$ respectively.

\subsection{A priori bounds}

In this section we collect some coarse bounds that will be useful in the sequel.
To compare discrete and continuous measures, we consider the canonical embedding $\iota_\cT : \cP(\cT) \to \cP(\cS^1)$ defined by
\begin{align*}
	\iota_{\cT} m  & = \sum_{k=0}^{K-1} m_k \cU_{A_k}
		 \qquad \text{for } m \in \cP(\cT) \ ,
\end{align*}
where $\cU_{A_k}$ denotes the uniform probability measure on $A_k$.
Note that $\iota_{\cT}$ is a right-inverse of the projection map $P_{\cT}$ defined by \eqref{eq: def of projection map}. We will often write $\iota_N = \iota_{\cT_N}$ for brevity.

\medskip

The following notion of mesh regularity can be found in a multi-dimensional setting in \cite[Section 3.1.2]{Eymard-Gallouet-Herbin:2000}.

\begin{definition}[$\zeta$-regularity]
Let $\zeta \in (0,1]$. We say that a mesh $\cT$ is $\zeta$-regular, if $\zeta < \frac{\min_k \{ z_k - r_k, r_{k+1} - z_k \}}{\max_k \pi_k}$ for all $k = 0, \ldots, K-1$.
\end{definition}

A mesh $\cT$ is $\zeta$-regular if and only if the ball of radius $\zeta \max_k \pi_k$ around $z_k$ is contained in the interior of the cell $A_k$ for each $k$.
Clearly, any mesh $\cT$ on $\cS^1$ is $\zeta$-regular for some $\zeta \in (0,1]$.

\begin{remark}\label{rem:reg}
If $\cT$ is $\zeta$-regular for some $\zeta \in (0,1]$, then each $\cT_N$ is $\zeta$-regular as well.
\end{remark}

Let $[\cT]$ denote the size of the mesh, i.e., the maximal diameter of its cells:
\begin{align*}
	[\cT] := \max\{\pi_k : k=0, \ldots, K-1 \} \ .
\end{align*}
The following result provide a coarse upper bound for $\cW_N$ in terms of $\bW$.

\begin{proposition}[Coarse upper bound for $\cW_\cT$]\label{prop:rough-upper-bound}
Let $\zeta \in (0,1]$.
There exists a constant $C < \infty$ depending only on $\zeta$ such that for any $\zeta$-regular mesh $\cT$ of $\cS^1$ and all $m_0, m_1 \in \cP(\cT)$ we have
\begin{align}
    \cW_\cT(\emm_0, \emm_1)
    	\leq C \Big( \bW(\iota_\cT \emm_0, \iota_\cT \emm_1) + [\cT] \Big) \ .
    \end{align}
\end{proposition}

\begin{proof}
This result has been proved in \cite[Lemma 3.3]{Gladbach-Kopfer-Maas:2018} for convex domains in $\R^d$; the proof on $\cS^1$ proceeds \emph{mutatis mutandi}.
\end{proof}

The following result provides a coarse bound in the opposite direction.

\begin{proposition}[Coarse lower bound for $\cW_\cT$]
\label{prop:rough-lower-bound}
Fix $\delta, \zeta \in (0,1)$. There exists a constant $C < \infty$ depending only on $\delta$ and $\zeta$, such that for any $\zeta$-regular mesh $\cT$ on $\cS^1$, and for any solution $(m_t, J_t)_t$ to the discrete continuity equation \eqref{eq:cont-eq} satisfying
$\delta \leq \frac{m_t(k)}{\pi_k} \leq \delta^{-1}$
for all $t \in [0,1]$ and $k = 0, \ldots K-1$, we have
\begin{equation}
\bW^2( \iota_{\cT} m_0, \iota_{\cT} m_1) \leq C \int_0^1 \cA_{\cT}(m_t, J_t) \dd t\ .
\end{equation}
\end{proposition}

\begin{proof}
We define
\begin{align*}
		\mu_t = \iota_\cT m_t \ , \qquad
		j_t(x) = \frac{r_{k+1} - x}{\pi_k} J_t(k-1,k)
					 + \frac{x - r_k}{\pi_k} J_t(k,k+1) \ ,
\end{align*}
for $x \in A_k$. It follows that $(\mu_t, j_t)$ solves the continuous continuity equation. Moreover,
\begin{align*}
		\bA(\mu_t, j_t)
		 & = \sum_{k = 0}^{K-1} \frac{\pi_k}{m_t(k)}
				\int_{r_k}^{r_{k+1}}
				 	\bigg(  \frac{r_{k+1} - x}{\pi_k} J_t(k-1,k)
							 + \frac{x - r_k}{\pi_k} J_t(k,k+1)  \bigg)^2
							 \dd x
		 \\& \leq \frac12 \sum_{k = 0}^{K-1} \frac{\pi_k^2}{m_t(k)}
							 		 \bigg( J_t^2(k-1,k)
							 					+ J_t^2(k,k+1)  \bigg)
		 \\& = \frac12 \sum_{k = 0}^{K-1} J_t^2(k,k+1)
							 		 \bigg( \frac{\pi_k^2}{m_t(k)}
							 					+ \frac{\pi_{k+1}^2}{m_t(k+1)}   \bigg) \ .
\end{align*}
Write $\rho_t(k) = \frac{m_t(k)}{\pi_k}$.
In view of the bounds on $\emm_t(k)$, we have
\begin{align*}
	\theta_{k,k+1}(\rho_t(k), \rho_t(k+1))
	  & \leq \max \{\rho_t(k), \rho_t(k+1)\}
	\\& \leq \delta^{-2} \min \{\rho_t(k), \rho_t(k+1)\}
		  \leq 2\delta^{-2} \bigg(
			 	\frac{1}{\rho_t(k)} +  \frac{1}{\rho_t(k+1)} \bigg)^{-1} \ \ .
\end{align*}
Since $2 \zeta [\cT] \leq d_{k,k+1}$, we have
\begin{align*}
	\frac{\pi_k^2}{m_t(k)}
		 + \frac{\pi_{k+1}^2}{m_t(k+1)}
		 \leq [\cT] \bigg(\frac{\pi_k}{m_t(k)}
	 		 + \frac{\pi_{k+1}}{m_t(k+1)} \bigg)
		\leq \frac{1}{\zeta \delta^2} \frac{d_{k,k+1}}{\theta_{k,k+1}(\rho_t(k), \rho_t(k+1))} \ .
\end{align*}
It follows that
\begin{align*}
	\bA(\mu_t, j_t)
	& \leq \frac{1}{2 \zeta \delta^2}
	 			\sum_{k = 0}^{K-1} d_{k,k+1}
								\frac{J_t^2(k,k+1)}
										{\theta_{k,k+1}(\rho_t(k), \rho_t(k+1))}
 =	\frac{1}{2 \zeta \delta^2} \cA_\cT(m_t, J_t) \ ,
\end{align*}
which implies the result with $C = \frac{1}{2 \zeta \delta^2}$.
\end{proof}

\section{A simple proof of the lower bound in the 2-periodic case}
\label{sec:simple}

In this section we focus on the simplest non-trivial periodic setting, which corresponds to taking $K = 2$ in Figure \ref{fig:mesh-1-intro}. 
In this setting we present a short proof of the lower bound in Theorem \ref{thm:main} by connecting the problem to known results from \cite{Gigli-Maas:2013,Gladbach-Kopfer-Maas:2018}. This approach does not appear to generalise to $K \geq 3$. 

We fix a parameter $r \in (0, 1)$, and consider the mesh $\cT$ with $K = 2$, and $r_0 = 0$, $r_1 = r$, and $r_2 = 1$. To be able to apply the simple argument in this section, we define the points $z_0 = \frac{r}{2}$ and $z_1 = \frac{r+1}{2}$ to be the midpoints of the cells, so that $d_{01} = d_{12} = \frac{1}{2}$.

\begin{figure}[ht]
\centering
\begin{tikzpicture}[xscale=0.8]
\draw [thick] (0,0)  -- (12,0);

\draw (0,-.2)  -- (0, .2) ; \node [below] at (0, -.4) {$0$};
\draw[thick] (1,.1)  -- (1,-.1) ; \node [below] at (1,-.25) {$\displaystyle \tfrac{r}{N}$};
\draw[red, fill] (0.5,0)  circle [radius=0.100] ;
\draw (3,-.2) -- (3, .2); \node [below] at (3, -.25) {$\displaystyle \tfrac{1}{N}$};
\draw[red, fill] (2,0) circle [radius=0.100];

\draw[thick] (4,.1)  -- (4,-.1) ; \node [below] at (4,-.25) {$\displaystyle \tfrac{1+r}{N}$};
\draw[red, fill] (3.5,0)  circle [radius=0.100] ;
\draw (6,-.2) -- (6, .2); \node [below] at (6, -.25) {$\displaystyle \tfrac{2}{N}$};
\draw[red, fill] (5,0) circle [radius=0.100];

\draw[thick, loosely dotted] (8.5, -0.7) -- (9.5, -0.7) ; 

\draw (12, -.1) -- (12, .1); \node [below] at (12, -.4) {$1$};

\end{tikzpicture}
 \caption{A $2$-periodic mesh $\cT_N$ on $\cS^1$.}
\label{fig:mesh-2periodic}
\end{figure}

Throughout this section we make the standing assumption that $\theta_{01} = \theta_{21}$, and we simply write $\theta := \theta_{01}$. 
This implies that the constant $c^\star(\theta, r) := c^\star(\theta, \cT)$ is given by 
\begin{equation} \label{eq: hom-constant for 2-periodic case}
	c^\star(\theta, r) 
		= \inf_{\alpha \in [0,1]} 
			\frac{1}{\theta\big(\frac{\alpha}{r}, 
			\frac{1-\alpha}{1-r}\big)}\ .
\end{equation}
The notation $m(k) = m(A_k)$ allows us to canonically identify measures on $\cT_N$ with measures on the equidistant mesh corresponding to $r = \frac12$. 
We write $\cT_{r,N}$ to emphasise the dependence of $\cT_N$ on $r$, and write $\cA_{r, N}^\theta$ and $\cW_{r, N}^\theta$ to denote the corresponding energy and metric. The cells in $\cT_{r,N}$ will be labeled $0, \ldots, 2N - 1$.

 \subsection{Lower bound}

The following lemma compares the discrete transport metric on the mesh $\cT_{r,N}$ with the corresponding quantity on the equidistant mesh $\cT_{\frac{1}{2},N}$.

\begin{lemma} \label{lemma: lower bound}
Let $r \in (0,1)$ and $N \geq 1$. For any $m_0, m_1 \in \cP(\cT_{r,N})$ we have
\begin{equation} \label{eq: lower bound claim}
	\cW^\theta_{r,N}(m_0, m_1)
	\geq \sqrt{c^\star(\theta,r)}\; \cW^{\theta_a}_{\frac12, N}(m_0, m_1)
\end{equation}
where $\theta_a$ denotes the arithmetic mean.
\end{lemma}

\begin{proof}
Note that
\begin{align*}
	\cA_{r,N}^\theta(m,J)
	= \frac{1}{2N^2} \sum_{n=0}^{N-1} \bigg\{
		\frac{(J(2n, 2n+1))^2}
			 {\theta\big(\frac{m(2n)}{r}, \frac{m(2n+1)}{1-r}\big)}
			 +
		\frac{(J(2n-1, 2n))^2}
			 {\theta\big(\frac{m(2n)}{r},\frac{m(2n-1)}{1-r}\big)}
			 \bigg\} \ .
\end{align*}
The key observation is that the mean $\theta$ of the densities $\frac{m(2n)}{r}$ and $\frac{m(2n\pm1)}{1-r}$ on the mesh $\cT_{r,N}$ can be estimated in terms of the \emph{arithmetic mean} $\theta^a$ of the corresponding densities $\frac{m(2n)}{1/2}$ and $\frac{m(2n\pm1)}{1/2}$ on the symmetric mesh $\cT_{\frac12, N}$. Indeed, as we can write
\begin{align*}
	m(2n) = \alpha_n^\pm\big( m(2n) + m(2n \pm 1) \big) \tand
	m(2n \pm 1) = (1-\alpha_n^\pm) \big(m(2n) + m(2n \pm 1) \big)
\end{align*}
for some $\alpha_n^{\pm} \in [0,1]$, the $1$-homogeneity of $\theta$ yields
\begin{align*}
	\theta\bigg(\frac{m(2n)}{r}, \frac{m(2n \pm 1)}{1-r}\bigg)
		& =  \big( m(2n) + m(2n \pm 1) \big)
				\theta\bigg(\frac{\alpha_n^\pm}{r},
					     \frac{1-\alpha_n^\pm}{1-r}\bigg)
	\\& \leq \theta^a\bigg(\frac{m(2n)}{1/2},
						   \frac{m(2n\pm1)}{1/2}\bigg)\frac{1}{c^\star(\theta, r)} \ .
\end{align*}
Consequently,
\begin{align*}
		\cA_{r,N}^\theta(m,J)
		\geq c^\star(\theta, r) \cA_{\frac12,N}^{\theta^a}(m,J) \ .
\end{align*}
As the continuity equation does not depend on $r$, this implies the result.
\end{proof}

The sought lower bound for $\cW_{r,N}^\theta$ can now be easily obtained.

\begin{corollary} \label{corollary: final lower bound}
Fix $r \in (0,1)$. 
For any $\mu_0, \mu_1 \in \cP(\cS^1)$, we have
\begin{align*}
\liminf_{N \to \infty} \cW_{r,N}(P_N \mu_0, P_N \mu_1)
	\geq \sqrt{c^\star(\theta,r)} \bW(\mu_0,\mu_1) \ .
\end{align*}
\end{corollary}

\begin{proof}
This follows by applying Lemma \ref{lemma: lower bound} to the measures $m_i := P_N \mu_i$ and using the known convergence result for symmetric meshes \cite{Gigli-Maas:2013,Gladbach-Kopfer-Maas:2018}, which asserts that
\begin{align*}
\lim_{N \to \infty}
	\cW^{\theta_a}_{\frac12,N}
		(P_N \mu_0, P_N \mu_1) = \bW(\mu_0,\mu_1) \ .
\end{align*}
\end{proof}

For proving the corresponding upper bound
\begin{align*}
	\limsup_{N \to \infty} \cW_{r,N}(P_N \mu_0, P_N \mu_1) \leq \sqrt{c^\star(\theta,r)} \bW(\mu_0,\mu_1) \ ,
\end{align*} 
the $2$-periodic setting does not offer conceptual simplifications compared to the general $K$-periodic setting. Therefore, we will directly treat the $K$-periodic setting in Section \ref{sec:upper}.

\subsection{Examples}

We finish this section by explicitly computing the value of $c^\star(\theta,r)$ in a number of cases.
We write 
\begin{align*}
	g_{\theta,r}(\alpha) := \theta\bigg( \frac{\alpha}{r}, \frac{1 - \alpha}{1-r}\bigg) 
	\ , \quad \text{ so that } \quad
	c^\star(\theta,r) = \inf_{\alpha \in (0,1)} \frac{1}{g_{\theta,r}(\alpha)} 
	 \ .
\end{align*}

\begin{example}[$\theta$ is $r$-balanced]\label{ex:}
Suppose that $\theta(a,b) \leq r a + (1-r) b$ for any $a, b \geq 0$ (i.e., $\theta$ is $r$-balanced in the sense of Definition \ref{def:balance} below). 
Applying this inequality to $a = \frac{\alpha}{r}$ and $b = \frac{1 - \alpha}{1-r}$ we immediately obtain $g_{\theta, r}(\alpha) \leq 1$ for all $\alpha \in [0,1]$. Since $g_{\theta, r}(r) = \theta(1,1) = 1$ by assumption, it follows that
\begin{align*}
	c^\star(\theta,r) = \frac{1}{g_{\theta, r}(r)} = 1 \ .
\end{align*}
\end{example}

\begin{example}[Geometric mean]\label{ex:geometric}
Let  $\theta(a,b) = \sqrt{ab}$. Then $g_{\theta,r}(\alpha) =
\sqrt{\frac{\alpha(1 - \alpha)}{r(1-r)}}$ is uniquely maximised at $\alpha = \frac12$, and we obtain
\begin{align*}
	c^\star(\theta,r) = 2\sqrt{r(1-r)} \ .
\end{align*}
Note that the fact that $\alpha^\star = \frac12$ means that the mass is equally distributed among large and small cells, irrespectively of the value of $r$. Thus, there will be no oscillations for the optimal discrete measures; however, this means that oscillations at the level of the density do occur.
\end{example}

\begin{example}[Harmonic mean]\label{ex:harmonic}
Let $\theta(a,b) = \frac{2ab}{a + b}$. In this case we have
$
	g(\alpha) :=  \frac{1}{g_{\theta,r}(\alpha)}
		=  \frac12 \big(\frac{1-r}{1-\alpha} + \frac{r}{\alpha} \big) \ ,
$
and $g'(\alpha) = \frac{1-r}{2(1-\alpha)^2} - \frac{r}{2\alpha^2}$. It follows that $g'$ vanishes at $\alpha^\star = \frac{\sqrt{r}}{\sqrt{r} + \sqrt{1-r}}$, which is indeed the unique minimiser of $g$. Consequently, 
\begin{align*}
	c^\star(\theta,r) = g(\alpha^\star) =  \frac12\big(\sqrt{r} + \sqrt{1-r}\big)^2 \ .
\end{align*}
\end{example}

\begin{example}[Arithmetic mean]\label{ex:arithmetic}
Let $\theta(a,b) = \frac{a+b}{2}$. Then $g_{\theta,r}(\alpha) = \frac12\big( \frac{1 - \alpha}{1-r} +  \frac{\alpha}{r}\big)$ is affine in $\alpha$. 
If $r < \frac12$ (resp. $r > \frac12$), the maximum is attained at $\alpha^\star = 1$ (resp. $\alpha^\star = 0$). In both cases, this means that all the mass will be assigned to the small cells.
It follows that
\begin{align*}
	c^\star(\theta,r) = 2 \min\{ r,  1-r \} \ .
\end{align*}
\end{example}

\begin{example}[Minimum]\label{ex:minimum}
Let $\theta(a,b) = \min \{ a, b \}$.
In this case, $g_{\theta,r}(\alpha) = \min\{ \frac{1 - \alpha}{1-r}, \frac{\alpha}{r} \}$ is uniquely maximised at $\alpha^\star = r$. This means that the assigned mass is proportional to the size of the cells, hence there are no oscillations at the level at the density. We find
\begin{align*}
	c^\star(\theta,r) = \frac{1}{g_{\theta,r}(\alpha^\star)} = 1 \ .
\end{align*}
\end{example}

\section{Analysis of the effective mobility}
\label{sec:c-star}
In this section we investigate some basic properties of the effective mobility $c^\star(\theta, \cT)$ defined in \eqref{eq: hom-constant for K-periodic case}, and relate its value to certain geometric properties of the mesh $\cT$ that have been considered in \cite{Gladbach-Kopfer-Maas:2018}. Recall:
\begin{equation} \label{eq:c-again}
c^\star(\theta, \cT) := \inf \left \{  \sum_{k=0}^{K-1} \frac{d_{k,k+1}}{\theta_{k,k+1} \Big( \frac{m_k}{\pi_k} , \frac{m_{k+1}}{\pi_{k+1}} \Big)} \suchthat 
m \in \cP(\cT)
 \right \} \ .
\end{equation}

We start with a simple observation.

\begin{proposition}\label{prop:c-less-than-one}
For any mesh $\cT$ on $\cS^1$ and any family of means $\theta = \{\theta_{k,k+1}\}_{k=0}^{K-1}$ we have $c^\star(\theta, \cT) \leq 1$.
\end{proposition}

\begin{proof}
This follows by using the competitor $m_k = \pi_k$ in \eqref{eq:c-again}.
\end{proof}

In view of this result, Theorem \ref{thm:main} implies the upper bound
\begin{align*}
	\limsup_{N \to \infty} 
		\cW_N(P_N \mu_0, P_N \mu_1) \leq 
		\bW(\mu_0,\mu_1) \ ,
\end{align*}
which had already been proved in \cite{Gladbach-Kopfer-Maas:2018}.

\begin{proposition}\label{prop:existence-of-min}
The infimum in \eqref{eq:c-again} is attained.
\end{proposition}

\begin{proof}
This readily follows using the lower-semicontinuity result from Lemma \ref{lem:prop_admissible_means}.
\end{proof}

In the remainder of this section we shall investigate under which conditions on $\theta$ and $\cT$ we have $c^\star(\theta, \cT) = 1$.
For this purpose, we consider two geometric conditions:

\begin{definition}[Geometric conditions on the mesh]\label{def:COM-and-ISO}
Fix $\{\lambda_{k,k+1}\}_{k=0}^{K-1} \in [0,1]^K$, and set $\lambda_{k+1,k} = 1 - \lambda_{k,k+1}$.
We say that a mesh $\cT = \cT_{\pi,z}$ on $\cS^1$ satisfies
\begin{enumerate}
\item the \emph{center-of-mass condition} with parameters $\{\lambda_{k,k+1}\}_{k=0}^{K-1}$ if, for all $k$,
\begin{align}
	\label{eq:com}
	r_{k+1} = \lambda_{k+1,k} z_k + \lambda_{k,k+1} z_{k+1}  \ ;
\end{align}
\item the \emph{isotropy condition} with parameters $\{\lambda_{k,k+1}\}_{k=0}^{K-1}$ if, for all $k$,
\begin{align}
	\label{eq:iso}
 \pi_k = \lambda_{k,k-1} d_{k-1, k}
		 + \lambda_{k,k+1} d_{k, k+1}   \ .
\end{align}
\end{enumerate}
\end{definition}

Both of these conditions have been studied for meshes on bounded convex domains in $\R^d$ in \cite{Gladbach-Kopfer-Maas:2018}. The center-of-mass condition asserts that the center of mass of the cell interfaces lie on the line segment connecting the support points of the respective cells. In dimensions $d \geq 2$, this condition poses a strong geometric condition on the mesh. However, in our one-dimensional context, the condition is always satisfied, for a unique choice of the parameters $\{\lambda_{k,k+1}\}_k$.
The isotropy condition is weaker than the center-of-mass condition: it holds with the same parameters, but there is an additional degree of freedom, as the following result shows.

\begin{proposition}\label{prop:com-iso-char}
Let $\cT = \cT_{\pi,z}$ be a mesh on $\cS^1$ and set $\overline\lambda_{k,k+1} = \frac{r_{k+1} - z_k}{z_{k+1} - z_k}$ for $k = 0, \ldots, K-1$. 
For $\{ \lambda_{k,k+1} \} \subseteq [0,1]$ the following assertions hold:
\begin{enumerate}
\item 
The center-of-mass condition holds 
if and only if for any $k = 0 , \ldots , K-1$,
\begin{align*}
	 \lambda_{k,k+1} = \overline\lambda_{k,k+1} \ .
\end{align*}
\item
The isotropy condition holds 
if and only if there exists 
$s \in [ - \min_k \overline\lambda_{k,k+1} d_{k,k+1}, \\ \min_k  \overline\lambda_{k+1,k} d_{k,k+1}]$ such that for any $k = 0 , \ldots , K-1$, 
\begin{align*}
	\lambda_{k,k+1} = \overline\lambda_{k,k+1} + \frac{s}{d_{k,k+1}} \ .
\end{align*}
\end{enumerate}
\end{proposition}

\begin{proof}
This follows immediately by solving the corresponding linear systems.
\end{proof}

\begin{remark}[Relation to the asymptotic isotropy condition]\label{rem:asymp-iso}
Recall from \cite[Definition 1.3]{Gladbach-Kopfer-Maas:2018} that a family of meshes $\{\cT\}$ (in any dimension) is said to satisfy the \textit{isotropy condition} with parameters $\{\lambda_{KL}\}$ if, for any $K \in \cT$,
\begin{equation} \label{eq: general isotropy condition K}
\sum_{L \in \cT} \lambda_{KL} \frac{\Haus^{d-1}(\partial K \cap \partial L)}{| z_K - z_L |}
	( z_K - z_L ) \otimes ( z_K - z_L ) 
	\leq |K| \big( I_d + \eta_{\cT}(K) \big)
\end{equation}
where $\displaystyle \sup_{K \in \cT} | \eta_{\cT}(K)| \to 0$ as $\max \{\diam(A) \, : \, A \in \cT \}  \to 0$.

Applying this condition to the family of one-dimensional periodic meshes $\cT^N$ constructed from $\cT$, it reduces to
\begin{equation*}
	\lambda_{k,k-1} d_{k-1,k} + \lambda_{k,k+1} d_{k,k+1}
		\leq \pi_k ( 1 + \eta_N(k) )
\end{equation*}
for all $N\geq 1$ and $k = 0, \ldots, K-1$, where $\eta_N(k) \to 0 $ as $N \to \infty$.
As the left-hand side does not depend on $N$, this condition in turn simplifies to
\begin{equation} \label{eq: isotropy dim 1 ineq}
\lambda_{k,k-1} d_{k-1,k} + \lambda_{k,k+1} d_{k,k+1} \leq \pi_k
\end{equation}
for all $k=0, \ldots, K-1$.

Clearly, \eqref{eq:iso} implies \eqref{eq: isotropy dim 1 ineq}.
To see that both assertions are equivalent, we note that \eqref{eq: isotropy dim 1 ineq} can be written as
\begin{equation} \label{eq: isotropy dim 1 ineq2}
\lambda_{k,k+1} d_{k,k+1} - \lambda_{k-1,k} d_{k-1,k} \leq \pi_k - d_{k-1,k} \ .
\end{equation}
To obtain a contradition, suppose that we have strict inequality in \eqref{eq: isotropy dim 1 ineq2} for some $k = \bar{k}$. Summation over $k=0, \ldots, K-1$ yields
\begin{align*}
	0 = 
		\sum_{k=0}^{K-1} \big( \lambda_{k,k+1} d_{k,k+1}
							 - \lambda_{k-1,k} 	d_{k-1,k} \big)
	  < \sum_{k=0}^{K-1} \big( \pi_k - d_{k-1,k} \big)
	  = 0 \ ,
\end{align*}
which is absurd.

In summary, we conclude that the isotropy condition \eqref{eq:iso} is equivalent to the asymptotic isotropy condition \eqref{eq: general isotropy condition K} for the family of meshes $\{\cT^N\}$.
\end{remark}

The next definition will be used to connect geometric properties of the mesh to properties of the means in the definition of the transport distance.

\begin{definition}[Adaptedness]\label{def:balance}
Let $\lambda, \lambda_{k,k+1} \in [0,1]$ for $k = 0, \ldots, K-1$.
\begin{enumerate}
\item  A mean $\theta$ is said to be \emph{$\lambda$-balanced} if $\theta(a,b) \leq \lambda a + ( 1-\lambda) b$ for any $a, b \geq 0$.
\item A family of means $\{\theta_{k,k+1}\}$ is said to be \emph{adapted} to the parameters $\{\lambda_{k,k+1}\}$ if $\theta_{k,k+1}$ is $\lambda_{k,k+1}$-balanced for each $k$.
\end{enumerate}
\end{definition}

\begin{remark}\label{rem:balance}
Each continuously differentiable mean $\theta$ is $\lambda$-balanced for exactly one value of $\lambda \in [0,1]$, namely
\begin{equation} \label{eq: smooth balancing}
\lambda  = \partial_1 \theta(1,1) \ .
\end{equation}
A non-smooth mean $\theta$ can be $\lambda$-balanced for multiple values of $\lambda$, e.g., the mean $(a,b) \mapsto \min\{a,b\}$ is  $\lambda$-balanced for any $\lambda \in [0,1]$.
\end{remark}

Now we are ready to state the main result of this section. The result is consistent with the main result in \cite{Gladbach-Kopfer-Maas:2018}, which asserts that the asymptotic isotropy condition is necessary (and essentially sufficient) for Gromov--Hausdorff convergence of the discrete transport distance to $\bW$.

\begin{theorem}[Isotropy is equivalent to $c^\star(\theta, \cT)=1$] \label{thm: equivalence isotropy condition}
Let $\{\theta_{k,k+1}\}$ be a family of means that are adapted to $\{\lambda_{k,k+1}\}$.
\begin{enumerate}
\item \label{it:iso-implies-c}
If $\cT$ satisfies the isotropy condition with parameters $\{\lambda_{k,k+1}\}$,
then $c^\star(\theta, \cT) = 1$.
\item \label{it:c-implies-iso}
Assume that each mean $\theta_{k,k+1}$ is continuously differentiable.
If $c^\star(\theta, \cT) = 1$, then $\cT$ satisfies the isotropy condition with parameters $\{\lambda_{k,k+1}\}$.
\end{enumerate}
\end{theorem}

\begin{remark}[Minimum mean]\label{rem:minimum}
In view of Proposition \ref{prop:com-iso-char}, every mesh $\cT$ satisfies the isotropy condition for a suitable choice of $\{\lambda_k\}$. Since  the minimum mean $(a,b) \mapsto \min\{a,b\}$ is $\lambda$-balanced for any value of $\lambda \in [0,1]$, it thus follows from Theorem \ref{thm: equivalence isotropy condition} that $c^\star(\theta, \cT) = 1 $ if $\theta_{k,k+1} = \min$ for each $k$.
\end{remark}

\begin{proof}
To prove \eqref{it:iso-implies-c}, take any sequence $\{m_k\}_k$ with $\sum_{k=0}^{K-1} m_k = 1$.
Using Jensen's inequality, the adaptedness, the periodicity, and the isotropy condition, we obtain
\begin{align*}
\sum_{k=0}^{K-1} \frac{d_{k,k+1}}{\theta_{k,k+1} \Big( \frac{m_k}{\pi_k} , \frac{m_{k+1}}{\pi_{k+1}}\Big)}
	& \geq \Bigg(
\sum_{k=0}^{K-1} d_{k,k+1} 
	\theta_{k,k+1} \bigg( \frac{m_k}{\pi_k}, 
					\frac{m_{k+1}}{\pi_{k+1}} \bigg) \Bigg)^{-1}
\\& \geq
	\Bigg(
	\sum_{k=0}^{K-1} d_{k,k+1}
	\bigg( \lambda_{k,k+1} \frac{m_k}{\pi_k}
		+ \lambda_{k+1,k} \frac{m_{k+1}}{\pi_{k+1}} \bigg)
	\Bigg)^{-1}
\\& =
	\Bigg(
	\sum_{k=0}^{K-1} \frac{m_k}{\pi_k}
	\big( \lambda_{k,k+1} d_{k,k+1}
		+  \lambda_{k,k-1} d_{k-1,k}   \big)
	\Bigg)^{-1}
\\&
=
	\Bigg(
	\sum_{k=0}^{K-1} m_k
	\Bigg)^{-1}
	= 1 \ .
\end{align*}
Taking the infimum over $\{m_k\}_k$, we obtain $c^\star(\theta, \cT) \geq 1$. In view of Proposition \ref{prop:c-less-than-one} we infer that $c^\star(\theta, \cT) = 1$.

To prove \eqref{it:c-implies-iso}, we consider the probability measures $\gamma_\alpha^k$ defined by
\begin{align*}
\gamma^k_\alpha = \left(
	\pi_0, \ldots ,
	\pi_{k-1},
	\pi_k + \alpha,
	\pi_{k+1} - \alpha,
	\pi_{k+2}, \ldots ,
	\pi_{K-1} \right)
\end{align*}
for $|\alpha|$ sufficiently small.
Let us write
\begin{align*}
	 h_{\theta,\cT}(m)
	 	= \sum_{k=0}^{K-1} \frac{d_{k,k+1}}{\theta_{k,k+1} \Big( \frac{m_k}{\pi_k} , \frac{m_{k+1}}{\pi_{k+1}} \Big)}  \ .
\end{align*}
As $c^\star(\theta, \cT) = 1$, we have $h_{\theta,\cT}(m) \geq 1$ for all $m$. Thus, since $h_{\theta,\cT}(\gamma^k_0) = h_{\theta,\cT}(\pi) = 1$, it follows that
$\frac{\rm d}{{\rm d} \alpha} \big|_{\alpha = 0} h_{\theta,\cT}(\gamma^k_\alpha) = 0$.
A direct computation shows that
\begin{align*}
\frac{\rm d}{{\rm d} \alpha}
	\bigg|_{\alpha = 0} h_{\theta,\cT}(\gamma^k_\alpha)
	=
	B_{k+1} - B_k
\quad \text{ where } \quad B_k := \frac{ \lambda_{k,k-1} d_{k-1,k} + \lambda_{k,k+1} d_{k,k+1}}{\pi_k} \ .
\end{align*}
As this holds for every $k$, we infer that there exists a constant $\beta > 0$ such that
$B_k = \beta$ for every $k = 0, \ldots, K-1$.
The latter means that
\begin{align*}
\beta \pi_k = \lambda_{k,k-1} d_{k-1,k} + \lambda_{k,k+1} d_{k,k+1}
\end{align*}
for all $k=0, \ldots, K-1$.
Summation over $k$ yields
\begin{align*}
\beta & = \beta \sum_{k=0}^{K-1} \pi_k
	 	= \sum_{k=0}^{K-1}
 	  	  	(1 - \lambda_{k-1,k} ) d_{k-1,k}
			 + \lambda_{k,k+1} d_{k,k+1}
	    = \sum_{k=0}^{K-1}
			 d_{k-1,k}
		= 1 \ ,
\end{align*}
which proves the isotropy condition with parameters $\{ \lambda_{k, k+1} \}_k$.
\end{proof}

\section{Proof of the lower bound}
\label{sec:lower}

The goal of this section is to prove Theorem \ref{thm:lowerbound_quantitative}, which yields the lower bound in Theorem \ref{thm:main}.
The crucial ingredient is Proposition \ref{prop: regularization at the discrete level}, which ensures the existence of approximately optimal curves with good regularity properties.

To formulate this result, we fix a non-negative function $\eta \in C_{\rm c}^\infty(0, \frac12)$ with $\int_0^1 \eta(x) \dd x = 1$.
We set $\eta_\lambda(x) = \frac{1}{\lambda} \eta(\frac{x}{\lambda})$ for $x \in [0,1)$, and consider its periodic extension to $\cS^1$.
For $\lambda \in (0,1]$ we define a discrete spatial mollifier by
\begin{align*}
\eta_\lambda^N(n) := \frac{N}{\lambda} \int_{\frac{n}{N}}^{\frac{n+1}{N}} \eta\Big(\frac{x}{\lambda}\Big) \dd x \ , \quad n = 0, \ldots, N - 1 \ ,
\end{align*}
and we extend $\eta_\lambda^N$ to $\Z$ periodically modulo $N$, so that it can be regarded as a function on the discrete torus $\bT_N = \Z / N \Z$.
It follows that $\frac{1}{N} \sum_{n=0}^{N-1} \eta_\lambda^N(n) = 1$, and the following kernel bounds hold for $n = 0, \ldots, N-1$:
\begin{align}
	\label{eq:kernel-bounds}
	 |\eta_\lambda^N(n)| \leq \frac{\|\eta\|_\infty}{\lambda} \ , \qquad
	 |\eta_\lambda^N(n_1) - \eta_\lambda^N(n_2) | \leq \frac{\|\eta'\|_\infty}{\lambda^2}\frac{|n_1 - n_2|}{N} \ ,
\end{align}
We consider the convolution operators $\bM_\lambda : L^1(\cS^1) \to L^\infty(\cS^1)$ given by
\begin{align*}
	\big(\bM_\lambda f \big)(x) = \int_{\cS^1} \eta_\lambda(x-y) f(y)  \dd y  \ ,
\end{align*}
as well as the analogous discrete convolution operators $\cM_\lambda^N : L^1(\bT_N) \to L^\infty(\bT_N)$ defined by
\begin{align*}
	\big(\cM_\lambda^N \psi\big)(n) = \frac{1}{N} \sum_{j=0}^{N-1} \eta_\lambda^N(n-j) \psi(j) \ .
\end{align*}
The kernel bounds \eqref{eq:kernel-bounds} imply the following $L^1$-$L^\infty$ and $L^1$-Lipschitz bounds:
\begin{align}
\label{eq:L1-Linf-bound}
	\sup_n |\cM_\lambda^N \psi(n)| & \leq \frac{\|\eta\|_\infty}{\lambda N}
					\sum_{n=0}^{N-1} |\psi(n)| \ , \\
\label{eq:L1-Lip-bound}
	\sup_n |\cM_\lambda^N \psi(n_1) - \cM_\lambda^N \psi(n_2)|
		& \leq \frac{\|\eta'\|_\infty}{\lambda^2} \frac{|n_1 - n_2|}{N^2}
					\sum_{n=0}^{N-1} |\psi(n)| \ .
\end{align}

The following result contains some basic properties of convolution operators that will be used in the sequel. 

\begin{lemma}[Bounds for convolution operators]\label{lem:commutator}
Let $\lambda \in (0,1]$ and $N \geq 2$. For any $\mu \in \cP(\cS^1)$ and $m \in \cP(\bT_N)$ we have
\begin{align}
\label{eq:conv-bound-1}
	\bW(\mu, \bM_\lambda \mu) & \leq C \lambda \ , \\
\label{eq:conv-bound-2}
	\bW( \iota_N \cM_\lambda^N m, \bM_\lambda \iota_N  m )
			& \leq \frac{\lambda}{2} + \frac{2}{N} \ ,
\end{align}
where $C < \infty$ depends only on $\eta$.
\end{lemma}

\begin{proof}
The inequality \eqref{eq:conv-bound-1} follows straightforwardly using the coupling $\gamma(\ddd x, \ddd y) = \eta_\lambda(y-x) \dd \mu(x) \dd y$. 

To prove \eqref{eq:conv-bound-2}, let $\delta_i$ be the Dirac mass at $i$, and note that
\[
	\bW^2( \iota_N \cM_\lambda^N m, \bM_\lambda \iota_N  m )
	\leq
	\sum_{i=0}^{N-1} m_i
		\bW^2( \iota_N \cM_\lambda^N \delta_i, \bM_\lambda \iota_N  \delta_i) 
\]
by convexity of $\bW^2$.
Thus it suffices to prove the lemma for $m = \delta_i$. 
Since
	$\dd \big( \iota_N  \delta_i \big)(x) 
		= N \one_{\big[ \frac{i}{N}, \frac{i+1}{N}  \big]}(x) \dd x$, 
	we have
\begin{equation} \label{eq:MQdelta}
	\dd \big( \bM_\lambda \iota_N  \delta_i \big) (x) 
		= N \Big( \eta_{\lambda} 
			* \one_{\left[ \frac{i}{N}, \frac{i+1}{N} \right]} \Big) (x) \dd x \ .
\end{equation}
On the other hand, we have
\begin{align*}
 \label{eq:QMdelta}
	\dd \big( \iota_N \cM_\lambda^N \delta_i \big) (x) 
		& = \sum_{n=0}^{N-1} \eta_{\lambda}^N(n-i)
			\one_{\left[ \frac{n}{N}, \frac{n+1}{N} \right]} (x) \dd x \ . 
\end{align*}
Since $\supp \eta_{\lambda} \subset \big( 0 , \frac{\lambda}{2} \big)$, we obtain
\begin{align*} 
	\supp \bM_\lambda \iota_N  \delta_i 
		\subseteq \big[\tfrac{i}{N}, 
				\tfrac{i+1}{N} + \tfrac{\lambda}{2}\big]  \tand
	 \supp \iota_N \cM_\lambda^N \delta_i 
		\subseteq \big[\tfrac{i}{N}, \tfrac{i+1}{N} 
							  + \tfrac{\lambda}{2}\big] \ ,
\end{align*}
hence
\begin{align*}
	\diam\big(\supp ( \iota_N \cM_\lambda^N \delta_i ) 
		\cup \supp ( \bM_\lambda \iota_N  \delta_i ) \big)
		\leq \frac{\lambda}{2} + \frac{2}{N} \ .
\end{align*}
This easily yields the desired result.
\end{proof}

Before stating the crucial regularisation result, we formulate a lemma which asserts that we can decrease the energy at the discrete level by a suitable regularisation. Here it is crucial that the regularisation is performed by averaging the density at spatial locations $n K + k$ and $n' K + k$ that differ by a multiple of the period $K$. 
A ``naive'' regularisation consisting of locally averaging the density, without taking the periodic structure into account, would in general not decrease the energy.  
We emphasise that the operator $\cM_\lambda^N$ is understood to act on the variable $n$ in the result below, namely
\begin{align*}
	(\cM_\lambda^N m) (n;k) = \frac{1}{N} \sum_{j=0}^{N-1} \eta_\lambda^N(n-j) m(j;k). 
\end{align*}
With this notation we have the following result.
\begin{lemma}[Energy bound under periodic smoothing]\label{lem:contr}
Let $\lambda \in (0, 1]$. For any $m \in \cP(\cT_N)$ and any $J \in \cV(\cT_N)$ we have
\begin{align*}
		\cA_N(\cM_\lambda^N m,\cM_\lambda^N J) 
			 \leq
		\cA_N(m, J) \ .
\end{align*}
\end{lemma}

\begin{proof}
For brevity we write
\begin{align*}
	G_{k,k+1}(m,J,n) = \frac{d_{k,k+1}}{N} f_{k,k+1}\bigg( N \frac{m(n;k)}{\pi_k}, N \frac{m(n;k+1)}{\pi_{k+1}}, J(n;k,k+1) \bigg) \ .
\end{align*}
Applying Jensen's inequality to the jointly convex functions $f_{k,k+1}$ we obtain

\begin{align*}
	\cA_N(\cM_\lambda^N m , \cM_\lambda^N J) 
	&= \sum_{k=0}^{K-1} \sum_{n=0}^{N-1} 
		G_{k,k+1}( \cM_\lambda^N m, \cM_\lambda^N J, n)  \\
	& \leq \sum_{k=0}^{K-1} \sum_{n=0}^{N-1} \frac{1}{N} \sum_{\ell=0}^{N-1} \eta_{\lambda}^N (n-\ell) G_{k,k+1}(m,J,\ell)  \\
	& = \sum_{k=0}^{K-1} \sum_{\ell=0}^{N-1} \Bigg( \frac{1}{N} \sum_{n=0}^{N-1} \eta_{\lambda}^N (n-\ell) \Bigg)  G_{k,k+1}(m,J,\ell) \\
	& = \sum_{k=0}^{K-1} \sum_{\ell=0}^{N-1} G_{k,k+1}(m,J,\ell) = \cA_N(m,J) \ ,
\end{align*}
where we used that $\frac{1}{N} \sum_{n=0}^{N-1} \eta_{\lambda}^N(n) =1$.
\end{proof}

We are now ready to state the main regularisation result of this section. 
As we expect that (approximately) optimal densities exhibit oscillations, we cannot expect spatial regularity for such densities.
Nevertheless, the lemma above allows us to obtain a restricted form of regularity for such densities, in the sense that good Lipschitz bounds hold if one only compares values of the density at spatial locations $n K + k$ and $n' K + k$ that differ by a multiple of the period $K$.

Note that the vector field $J$ enjoys better regularity properties: in \eqref{eq: reg4} we even obtain a Lipschitz bound for neighbouring cells.

\begin{proposition}[Space-time regularisation] \label{prop: regularization at the discrete level}
Fix $N \geq 1$, and let $(m_t, J_t)_t$ be a solution to the discrete continuity equation \eqref{eq:cont-eq} in $\cP(\cT_N)$ satisfying
\begin{align*}
	A := \int_0^1 \cA_N(m_t, J_t) \dd t < \infty \ .
\end{align*}
Then, for any $\eps > 0$ there exists a solution $(\tilde m_t, \tilde J_t)_t$ to \eqref{eq:cont-eq} such that:
\begin{enumerate}
\item $\bW(\iota_N \tilde m_t, \iota_N m_t) \leq \eps + \frac{C}{N}$ for all $t \in [0,1]$, where $C < \infty$ depends only on $\cT$;
\item the following action bound holds:
\begin{equation} \label{eq: reg5}
\int_0^1 \cA_N(\tilde m_t, \tilde J_t) \dd t
		\leq \int_0^1 \cA_N(m_t, J_t) \dd t  \ ;
\end{equation}
\item the following regularity properties hold, for some constants $c_0, \ldots, c_5 < \infty$ depending on $\eps$ and $A$, but not on $N$:
\begin{subequations}
\begin{align}
\label{eq: reg0}
	c_0^{-1}
	\leq \min_{n,k} N \tilde m_t(n;k)
	\leq \max_{n,k} N  \tilde m_t(n;k)
	& \leq c_1 \ , \\
\label{eq: reg1}
	\sup_{t \in [0,1]} \max_{n,k}
		\big| N \partial_t \tilde m_t^{N}(n;k) \big|
		 & \leq c_2 \ , \\
\label{eq: reg2}
	\sup_{t \in [0,1]} \max_{n,k}
		\big| N \tilde m_t(n;k) - N \tilde m_t(n+1;k) \big|
		 & \leq \frac{c_3}{N} \ , \\
\label{eq: reg3}
	\sup_{t \in [0,1]} \max_{n,k}
		\big| \tilde J_t^{N}(n;k,k+1) \big|
		 & \leq c_4 \ , \\
\label{eq: reg4}
	\sup_{t \in [0,1]} \max_{n,k}
		\big| \tilde J_t(n;k,k+1) - \tilde J_t(n;k-1,k) \big|
		 & \leq \frac{c_5}{N} \ .
\end{align}
\end{subequations}
\end{enumerate}
\end{proposition}

\begin{proof}
Let $\cU_N := P_N \Leb^1|_{\cS^1} \in \cP(\cT_N)$ denote the probability measure that assigns mass $\frac{\pi_k}{N}$ to $A_{n;k}$. Fix a mollifier $\eta$ as above.
For $\lambda, \tau, \delta > 0$ we define a space-time regularisation by
\begin{subequations}
\begin{align}
 \label{eq: regularized measure}
\tilde m_t(n;k) & := \frac{1}{2 \tau} \int_{t-\tau}^{t+\tau}
		\cM_\lambda^N \Big[ (1-\delta) m_u  + \delta \cU_N \Big] (n;k)
		 \dd u \ , \\
\label{eq: regularized vector field}
\tilde J_t(n;k,k+1)
	 & := \frac{1-\delta}{2 \tau} \int_{t-\tau}^{t+\tau}
	 	\cM_\lambda^N J_u (n;k,k+1) \dd u \ .
\end{align}
\end{subequations}
In both expressions, the operator $\cM_\lambda^N$ is understood to act on the variable $n$, i.e., the spatial averaging takes place over cells whose distance is an integer multiple of the period $K$. Moreover, we use the convention that $m_u = m_0$ and $J_u = 0$ for $u < 0$, and $m_u = m_1$ and $J_u = 0$ for $u > 1$. We claim that this approximation satisfies all the sought properties.

An explicit computation shows that $(\tilde m_t, \tilde V_t)_t$ solves the discrete continuity equation.

To prove \eqref{eq:  reg5}, we note that by a trifold application of the joint convexity of $\cA_N$,
\begin{align*}
 \cA_N(\tilde m_t, \tilde J_t)
	& \leq
	  \frac{1}{2 \tau} \int_{t-\tau}^{t+\tau}
		\cA_N\Big( \cM_\lambda^N \big[ (1-\delta) m_u  + \delta \cU_N \big],
			(1-\delta) \cM_\lambda^N J_u \Big) \dd u
\\	& \leq
	  \frac{1}{2 \tau} \int_{t-\tau}^{t+\tau}
		\cA_N\Big( \big[ (1-\delta) m_u  + \delta \cU_N \big],
			(1-\delta) J_u \Big) \dd u
\\	& \leq
	  \frac{1-\delta}{2 \tau} \int_{t-\tau}^{t+\tau}
		\cA_N( m_u, J_u ) \dd u \ .
\end{align*}
Here we used the crucial regularisation bound from Lemma \ref{lem:contr}. 
 The desired inequality \eqref{eq:  reg5} follows.

Moreover, since $\cM_\lambda^N$ preserves positivity, we deduce the lower bound in \eqref{eq: reg0} with $c_0^{-1} = \delta \min_{k} \pi_k$.

To prove the upper bound in \eqref{eq: reg0}, we use the fact that $m_t$ is a probability measure and the $L^1$-$L^{\infty}$ bound \eqref{eq:L1-Linf-bound} to obtain
\begin{align*}
	N \tilde m_t(n;k) \leq \frac{\| \eta \|_{\infty}}{\lambda}=: c_1\ .
\end{align*}

To prove \eqref{eq: reg1}, we observe that
\begin{align*}
	\partial_t \tilde m_t = \frac{1-\delta}{2\tau} \cM_\lambda^N \big[ m_{t + \tau} - m_{t - \tau} \big] \ .
\end{align*}
Therefore, by another application of the $L^1$-$L^\infty$-bound in \eqref{eq:L1-Linf-bound}, we arrive at
\begin{align*}
	N |\partial_t \tilde m_t(n;k)|
		\leq \frac{\| \eta \|_\infty }{\tau \lambda}
			=: c_2 \ ,
\end{align*}
which proves \eqref{eq: reg1}.

The inequality \eqref{eq: reg4}, with $c_5 = c_2$, follows immediately from \eqref{eq: reg1} and the fact that $(\tilde m_t^{N}, \tilde J_t)_t$ solves the continuity equation.

Furthermore, since
\begin{align*}
	| \tilde m_t(n;k) - \tilde m_t(n+1;k) |
	\leq \sup_{s} | \cM_\lambda^N m_s(n;k) - \cM_\lambda^N m_s(n+1;k) |
	\leq \frac{\| \eta'\|_\infty}{\lambda^2 N^2} \ ,
\end{align*}
we obtain \eqref{eq: reg2} with $c_3 = \frac{\| \eta'\|_\infty}{\lambda^2}$, in view of the Lipschitz bound in \eqref{eq:L1-Lip-bound}.

Finally, to obtain the $L^\infty$-bound on the vector field \eqref{eq: reg3}, we use  \eqref{eq:L1-Linf-bound} again to infer
\begin{align*}
	\sup_{n,k} | \tilde J_t(n;k,k+1) |
		& \leq  \frac{1-\delta}{2\tau}	\int_{t - \tau}^{t + \tau}
				\sup_{n,k} |\cM_\lambda^N J_u(n;k,k+1)|	 \dd u
		\\& \leq \frac{\|\eta\|_\infty}{\lambda N} \frac{1-\delta}{2\tau} \int_{t - \tau}^{t + \tau}
				\sup_{k}
					\sum_{n=0}^{N-1} | J_u(n;k,k+1) | \dd u \ .
\end{align*}
Writing 
$\theta_{n;k,k+1} = \theta_{k,k+1}\Big( N \frac{m_u(n;k)}{\pi_k},
			  N \frac{m_u(n;k+1)}{\pi_{k+1}}\Big)$ 
for brevity, we infer that
\begin{align*}
\frac{1}{N} \bigg( \sum_{n,k} | J_u(n;k,k+1) | \bigg)^2
	 & \leq \bigg(\sum_{n,k} 
	 		\frac{d_{k,k+1}}{N}
	 		\frac{J_u^2(n;k,k+1)}{\theta_{n;k,k+1}}
					  \bigg)
			  \bigg( 
			  \sum_{n,k} 
			  \frac{\theta_{n;k,k+1}}{ d_{k,k+1} }
			  \bigg)
	\\ & = \cA_N(m_u, J_u)	 
			  \sum_{n,k} 
			  \frac{\theta_{n;k,k+1}}{ d_{k,k+1} }
			  		   \ .
\end{align*}
Using the bound $\theta_{k,k+1}(a,b) \leq a + b$ we obtain
\begin{align*}
	\sum_{n,k} \frac{\theta_{n;k,k+1}}{ d_{k,k+1} }
	\leq \frac{N}{\min_k d_{k,k+1}} 
		\sum_{n,k}\bigg(\frac{ m_u(n;k)}{\pi_k} + \frac{m_u(n;k+1)}{\pi_{k+1}}\bigg)
	\leq 2 B N \ ,
\end{align*}
where $B = (\max_{k} \pi_k^{-1}) (\max_k d_{k,k+1}^{-1})$.
Combining these bounds, we arrive at
\begin{align*}
	\sup_{n,k} | \tilde J_t(n;k,k+1) |
		& \leq 
		\frac{\|\eta\|_\infty\sqrt{2B}}{2\tau\lambda} 
			\int_{t - \tau}^{t + \tau}
				\sqrt{\cA_N(m_u, J_u)}
			\dd u 
		\\ & \leq 
		\frac{\|\eta\|_\infty}{\lambda}
			\sqrt{ \frac{B}{\tau}
				  \int_0^1 \cA_N(m_u, J_u) \dd u }  \ ,
\end{align*}
which yields \eqref{eq: reg3} with $c_4 := \frac{\|\eta\|_\infty}{\lambda}\sqrt{ \frac{AB}{\tau}}$.
As we will choose $\delta, \lambda, \tau > 0$ depending on $\eps$, the bounds \eqref{eq: reg0}--\eqref{eq: reg4} follow.

\medskip
It remains to show that $\bW(\iota_N \tilde m_t, \iota_N m_t) \leq \eps + \frac{C}{N}$ for suitable values of $\delta, \lambda$ and $\tau$. We consider the effect of the three different regularisations separately. 
First we apply the convexity of $\bW^2$ to obtain
for any $m \in \cP(\cT_N)$,
\begin{align}\label{eq:close-1}
	\bW^2\big(\iota_N m, \iota_N [(1-\delta) m + \delta \cU_N] \big)
	\leq \delta \bW^2(\iota_N m, \Leb^1|_{\cS^1})
	\leq \frac{\delta}{4} \ ,
\end{align}
since the diameter of $(\cP(\cS^1), \bW)$ is equal to $\frac12$.
Moreover, for $m \in \cP(\cT_N)$, Lemma \ref{lem:commutator} yields
\begin{equation}\begin{aligned}\label{eq:close-2}
	\bW( \iota_N m, \iota_N \cM_\lambda^N m)
		& \leq   \bW( \iota_N m, \bM_\lambda \iota_N m)
			   + \bW( \bM_\lambda \iota_N m, \iota_N \cM_\lambda^N m)
		\\& \leq C \bigg( \lambda + \frac{1}{N}  \bigg)  \ ,
\end{aligned}\end{equation}
where $C < \infty$ depends only on $\eta$.
Furthermore, set $\bar m_t = \cM_\lambda^N \big( (1-\delta) m_t + \delta \cU_N \big)$  and $\bar J_t =  (1-\delta) \cM_\lambda^N J_t$. It then follows that $c_0 \leq N \bar m_t \leq c_1$ and $\int_0^1 \cA_N(\bar m_t, \bar J_t) \dd t \leq A$.
Thus, for $s \leq t$, Proposition \ref{prop:rough-lower-bound} yields a constant $\kappa < \infty$ depending on $c_0$ and $c_1$ (hence on $\delta$ and $\lambda$) such that,
\begin{align*}
	\bW^2(\iota_N \bar m_s, \iota_N \bar m_t)
		& \leq \kappa \int_0^1
			 \cA(\bar m_{(1-a)s + at}, (t-s) \bar J_{(1-a)s + at}) \dd a
		\\& \leq \kappa (t-s) \int_s^t 
			\cA(\bar m_u, \bar J_u) \dd u
		\\& \leq \kappa A (t-s) \ .
\end{align*}
By convexity of $\bW^2$, we obtain
\begin{align}
\label{eq:close-3}
	\bW^2\bigg(\iota_N  \bar{m} _t, \iota_N \bigg[\frac{1}{2 \tau} \int_{t - \tau}^{t + \tau} \bar{m} _u \dd u\bigg] \bigg)
	\leq \frac{1}{2 \tau} \int_{t - \tau}^{t + \tau} \bW^2(\iota_N \bar{m} _t, \iota_N \bar{m} _u) \dd u
	\leq  \frac{\kappa \tau A}{2} \ .
\end{align}

Applying the estimates \eqref{eq:close-1} with $m=m_t$, \eqref{eq:close-2} with $m= (1-\delta)m_t + \delta \cU_N$ and \eqref{eq:close-3}, we arrive at
\begin{align*}
	\bW\big( \iota_N \tilde{m}_t, \iota_N m_t \big)
	\leq C \bigg( \sqrt{\delta} + \lambda + \frac1N 
			+  \sqrt{\kappa \tau A} \bigg) \ ,
\end{align*}
for some $C < \infty$ depending only on $\cT$ and on $\eta$.
Thus, choosing first $\lambda$ and $\delta$ sufficiently small, and then $\tau$ sufficiently small depending on $\delta$,  the result follows.
\end{proof}

We are now ready to prove the lower bound in Theorem \ref{thm:main}.

\begin{theorem}[Lower bound for $\cW_N$]	\label{thm:lowerbound_quantitative}
For any mesh $\cT$ and any family of admisible means $\{\theta_{k,k+1}\}_k$ we have 
\begin{align*}
	c^\star(\theta,\cT) 
		\bW^2(\mu_0, \mu_1)
		\leq
		\liminf_{N \to \infty} \cW_N^2( P_N \mu_0, P_N \mu_1) \ ,
\end{align*}
uniformly for all $\mu_0, \mu_1 \in \cP(\cS^1)$. 
More precisely,  
for any $\eps > 0$ there exists $\bar N \in \N$ such that for any
$N \geq \bar N$ and $\mu_0, \mu_1 \in \cP(\cS^1)$, we have
\begin{equation} \label{eq:lowerbound_quantitative}
	c^\star(\theta,\cT) 
		\bW^2( \mu_0, \mu_1) 
			\leq 
		\cW_N^2(P_N \mu_0, P_N \mu_1) + \eps \ .
\end{equation}
\end{theorem}

\begin{proof}
Fix $\eps > 0$. Applying Proposition \ref{prop: regularization at the discrete level} to an approximate $\cW_N$-geodesic between $P_N \mu_0$ and $P_N \mu_1$, we infer that there exists a curve $(m_t, J_t)_t$ satisfying the bounds
	\begin{align}
			\label{eq:W2-iota-N-bound}
		\bW(\iota_N m_i, \iota_N P_N \mu_i) & \leq \eps + \frac{C}{N}  \quad \text{ for } i = 0,1 \ , \\ 
			\label{eq:action-cW-N-bound}
		\int_0^1 \cA_N(m_t, J_t) \dd t
		& \leq \cW_N^2(P_N \mu_0, P_N \mu_1) + \eps \ ,
	\end{align}
as well as the regularity properties \eqref{eq: reg0}--\eqref{eq: reg4}.

For brevity we write
\begin{align*}
\cA_N(m, J) = \frac{1}{N} \sum_{n=0}^{N-1} \cA_N^n(m, J) \ ,
\end{align*}
where
\begin{align*}
\cA_N^n(m,J)
	& = \sum_{k=0}^{K-1} d_{k,k+1}
		f_{k,k+1}\bigg(N \frac{m(n;k)}{\pi_k}, N \frac{m(n;k+1)}{\pi_k}, J(n;k,k+1) \bigg)  \ .
\end{align*}
We set
\begin{align*}
\hm_t(n) := \sum_{k=0}^{K-1} m_t(n;k) \ , \quad
 \hJ_t(n) := J_t(n;-1,0) \ ,
\end{align*}
and define $\alpha_t : \{0, \ldots, N-1\} \times \{0, \ldots, K \} \to \R$ by
\begin{align*}
	\alpha_t(n;k) = \frac{m_t(n;\sigma(k))}{\hm_t(n)} \quad \text{ for } k = 0, \ldots, K \ ,
\end{align*}
where $\sigma(k) = k$ for $k = 0, \ldots, K-1$, and $\sigma(K) = 0$.
Here it is important to note that $\alpha_t(n;K) \neq \alpha_t(n+1;0)$.
Observe that, for $k = 0, \ldots, K-1$,
\begin{align*}
	 \frac{1}{\theta_{k,k+1}\big(\tfrac{\alpha_t(n;k)}{\pi_k},\tfrac{\alpha_t(n;k+1)}{\pi_{k+1}}\big)}\frac{|\hJ_t(n)|^2}{N \hm_t(n)}
	 =
	f_{k,k+1}\bigg(N \frac{m_t(n;\sigma(k))}{\pi_k}, N \frac{m_t(n;\sigma(k+1))}{\pi_{k+1}},   \hJ_t(n)\bigg)
\ .
\end{align*}
Note that, for any $n$ and $k$, 
\begin{align*}
	\frac{c_0^{-1}}{\max_\ell \pi_\ell} 
		\leq N \frac{m_t(n;k)}{\pi_k}
		\leq 	\frac{c_1}{\min_\ell \pi_\ell}
		\tand
	|J_t(n;k,k+1)| \leq c_4 \ .
\end{align*}
Therefore, since the functions $f_{k,k+1}$ are Lipschitz on the set $[\frac{c_0^{-1}}{\max_\ell \pi_\ell}  , \frac{c_1}{\min_\ell \pi_\ell}]^2 \times [-c_4,c_4]$, it follows 
that, for $k = 0, \ldots, K-1$,
\begin{equation}\begin{aligned}\label{eq:Lip-bound}
	& \bigg| f_{k,k+1}\bigg(N \frac{m_t(n;k)}{\pi_k}, N \frac{m_t(n;k+1)}{\pi_{k+1}}, J_t(n;k,k+1) \bigg)
	 -
	 \frac{1}{\theta_{k,k+1}\big(\tfrac{\alpha_t(n;k)}{\pi_k},\tfrac{\alpha_t(n;k+1)}{\pi_{k+1}}\big)}\frac{|\hJ_t(n)|^2}{N \hm_t(n)} \bigg|
	\\&  \leq [f_{k,k+1}]_{\Lip}
	 \bigg( 
				\frac{N}{\pi_{k+1}} \big|  m_t(n;k+1) - m_t(n;\sigma(k+1)) \big|
				+ 			
			\big| J_t(n;k,k+1) - \hJ_t(n) \big|
				 \bigg)
	\\& \leq \frac{[f_{k,k+1}]_{\Lip}}{N} \bigg( 
			\frac{c_3}{\pi_{k+1}} 
			+  K c_5
			\bigg)
	 =: \frac{C}{N} \ ,
\end{aligned}\end{equation} 
for some $C < \infty$ depending on $\eps$ (through $c_0$, \ldots, $c_5$) and on $\cT$.

Since $\sum_{k=0}^{K-1} \alpha_t(n;k) = 1$, the sequence $\{ \alpha_t(n;k) \}_{k = 0}^{K-1}$ is, for any $n$, a competitor for the cell problem \eqref{eq: hom-constant for K-periodic case}. Taking into account that  $\alpha_t(n;0) = \alpha_t(n;K)$, it follows from \eqref{eq:Lip-bound} and the definition \eqref{eq: hom-constant for K-periodic case} of $c^\star(\theta, \cT)$ that
\begin{equation}\begin{aligned}\label{eq:A-n-bound}
\cA_N^n(m_t,J_t)
	& \geq \frac{|\hJ_t(n)|^2}{N \hm_t(n)} \sum_{k=0}^{K-1}
	\frac{ d_{k,k+1} }{\theta_{k,k+1}\big(\tfrac{\alpha_t(n;k)}{\pi_k},\tfrac{\alpha_t(n;k+1)}{\pi_{k+1}}\big)}
	- \frac{C}{N}
	\\ & \geq c^\star(\theta, \cT)
	\frac{|\hJ_t(n)|^2}{N \hm_t(n)}
	- \frac{C}{N} \ .
\end{aligned}\end{equation}

\medskip

At the continuous level, we define a curve of measures $(\mu_t^N)_t$ with piecewise constant densities, and a vector field $j_t^N$ by piecewise affine interpolation of $\hJ_t^N$; more precisely,
\begin{align*}
	\mu_t^N  & = \sum_{n=0}^{N-1} \hm_t \cU_{\hA_n} \ , \\
	j_t^N(x) & = \sum_{n=0}^{N-1} \chi_{\hA_n}(x) \Big[ (n + 1 - Nx) \hJ_t(n) + (Nx - n) \hJ_t(n+1)\Big] \ .
\end{align*}
As before, $\cU_{\hA_n}$ denotes the normalised Lebesgue measure on $\hA_n := [\frac{n}{N},\frac{n+1}{N})$.

We observe that the density $\rho_t^N$ of $\mu_t^N$ satisfies
\begin{align*}
\partial _t \rho_t^N (x)
		= N \sum_{k=0}^{K-1} \partial_t \hm_t(n;k)
		= N \big( \hJ_t(n) - \hJ_t(n+1) \big)
		= - \partial_x j_t^N(x)
\end{align*}
for any $x \in \big( \tfrac{n}{N}, \tfrac{n+1}{N} \big)$, which implies that $\big(\mu_t^N, j_t^N\big)_t$ solves the continuity equation.

To estimate the continuous energy, we find 
\begin{align*}
	\bA(\mu_t^N, j_t^N)
	& = \frac1N \sum_{n=0}^{N-1}	\frac{1}{\hm_t(n)}
			\int_{\frac{n}{N}}^{\frac{n+1}{N}}
				\Big[ (n + 1 - Nx) \hJ_t(n) + (Nx - n) \hJ_t(n+1)\Big]^2 \dd x
	\\& \leq
	\frac1{N} \sum_{n=0}^{N-1} \frac{\hJ_t(n)^2 + \hJ_t(n+1)^2}{2 N \hm_t(n)}
	\\& =
	\frac1{N} \sum_{n=0}^{N-1}
		\frac{\hJ_t(n)^2}{\theta_{\rm h}( N \hm_t(n), N \hm_t(n+1)) }
\end{align*}
where $\theta_{\rm h}(a,b) = \frac{2ab}{a+b}$ denotes the harmonic mean.
Note that 
\eqref{eq: reg2} implies the Lipschitz bound
\begin{align*}
	| N \hm_t(n) - N \hm_t(n-1)  | \leq \frac{K c_3}{N} \ .
\end{align*}
and \eqref{eq: reg0} yields a lower bound on the density: $ N \hm_t(n) \geq K c_0^{-1}$.
Furthermore, \eqref{eq: reg3} yields the estimate $|\hJ_t(n)| \leq c_4$.
Thus, in view of the identity
$
	\frac{1}{\theta_{\rm h}(a,b)}
	= \frac{1}{a} + \frac{a-b}{2ab}
$
we obtain
\begin{align} \label{eq:A-N-bound-cont}
	\bA(\mu_t^N, j_t^N)
	\leq \bigg(\frac1{N} \sum_{n=0}^{N-1}
		\frac{\hJ_t(n)^2}{N \hm_t(n)} \bigg)
		+ \frac{C}{N} \ ,
\end{align}
with $C < \infty$ depending on $\eps$ (through the $c_i$'s) and on $\cT$.

Putting things together, it follows from \eqref{eq:A-n-bound}, \eqref{eq:A-N-bound-cont} and
\eqref{eq:action-cW-N-bound} that
\begin{equation}\begin{aligned}\label{eq:bW-bound1}
		 c^\star(\theta, \cT) \bW^2(\mu_0^N, \mu_1^N)
	& \leq
		 c^\star(\theta, \cT) \int_0^1 \bA(\mu_t^N, j_t^N) \dd t
	\\& \leq  c^\star(\theta, \cT)
		\int_0^1 \frac1{N} \sum_{n=0}^{N-1}
		\frac{\hJ_t(n)^2}{N \hm_t(n)}
		 \dd t
				+ \frac{C}{N}
	\\& \leq
		\int_0^1 \frac1{N} \sum_{n=0}^{N-1}
		\cA_N^n(m_t,J_t)
		 \dd t
				+ \frac{C}{N}
	\\& = \int_0^1
		\cA_N(m_t,J_t)
		 \dd t
				+ \frac{C}{N}
	\\& \leq \cW_N^2(P_N \mu_0, P_N \mu_1) + \eps
				+ \frac{C}{N} \ .
\end{aligned}\end{equation}
Finally we note that, for $i = 0, 1$, \eqref{eq:W2-iota-N-bound} yields
\begin{align*}
	\bW(\mu_i,  \mu_i^N)
		& \leq 	\bW(\mu_i, \iota_N P_N \mu_i)
			+	\bW(\iota_N P_N \mu_i, \iota_N m_i)
			+   \bW(\iota_N m_i, \mu_i^N)
		\\& \leq \frac{1}{N} + \Big( \eps + \frac{C}{N} \Big) + \frac{1}{N}
		\\& \leq \eps + \frac{C}{N} \ ,
\end{align*}
which implies that
\begin{align}\label{eq:bW-bound2}
	\bW(\mu_0, \mu_1) \leq \bW(\mu_0^N, \mu_1^N)
		+ 2\Big(\eps + \frac{C}{N}\Big)  \ .
\end{align}
Combining \eqref{eq:bW-bound1} and \eqref{eq:bW-bound2} we obtain 
the desired result.
\end{proof}

\section{Proof of the upper bound}
\label{sec:upper}

In this section we present the proof of the upper bound for $\cW_N$.
The  idea of the proof of the upper bound is to start from optimal curves of measures at the continuous level, and to introduce the optimal oscillation in their discretations, as determined by the formula for the effective mobility \eqref{eq: hom-constant for K-periodic case}. 

Let $\alpha^\star = \{\alpha_k^\star\}_{k=0}^{K-1}$ be an optimiser in \eqref{eq: hom-constant for K-periodic case}, and define $P_N^\star : \cP(\cS^1) \to \cP(\cT_N)$ by
\begin{align}\label{eq:oscillating-projection}
	\big(P_N^\star \mu\big)(n;k)
	:= \alpha_k^\star \mu\big( \hA_n \big) \ , \qquad
\text{where \ }
	\hA_n := \Big[\frac{n}{N},\frac{n+1}{N}\Big) \ ,
\end{align}
as before. Slightly abusing notation, we also define $P_N^\star : C(\cS^1; \R) \to \cV(\cT_N)$ by
\begin{align*}
	\big(P_N^\star j\big)(n;k,k+1)
	:= \bigg(\sum_{\ell={k+1}}^{K-1} \alpha_\ell^\star\bigg) j\big(\tfrac{n}{N}\big)
	 + \bigg(\sum_{\ell=0}^{k} \alpha_\ell^\star\bigg) j\big(\tfrac{n+1}{N}\big)
		\ .
\end{align*}
Since $\sum_{k=0}^{K-1} \alpha_k^\star = 1$, the right-hand side is a convex combination of $j\big(\tfrac{n}{N}\big)$ and $j\big(\tfrac{n+1}{N}\big)$.

\begin{proposition}[Discretisation of the continuity equation]\label{prop:CE-cont-to-disc}
Let $(\mu_t)_{t \in [0,1]}$ be a Borel family of probability measures, and let $(j_t)_{t \in [0,1]}$ be a Borel family of continuous functions satisfying the continuity equation $\partial_t \mu + \dive j = 0$ on $\cS^1$.
Then the pair $(m_t, J_t)_{t \in [0,1]}$ defined by
\begin{align*}
	m_t := P_N^\star \mu_t \ , \quad
	J_t := P_N^\star j_t \ ,
\end{align*}
solves the continuity equation on $\cT_N$.
\end{proposition}

\begin{proof}
As $(\mu_t, j_t)_t$ satisfies the continuity equation, we have
\begin{align*}
	  \int_0^1 \bigg( \int_{\cS^1} \partial_t \phi_t(x) \dd \mu_t(x) 
	+ \int_{\cS^1} \partial_x \phi_t(x) j_t(x) \dd x \bigg) \dd t 
	= \int_{\cS^1} \phi_1(x) \dd \mu_1(x) - \int_{\cS^1} \phi_0(x) \dd \mu_0(x)  
\end{align*}
for any smooth function $\phi : [0,1] \times \cS^1 \to \R$.

Let $\psi : [0,1] \to \R$ be smooth, and define $\eta^\eps : \cS^1 \to \R$ by $\eta^\eps = \chi_{\hA_n} * \xi^\eps$ for a smooth mollifier $\xi^\eps$ supported in an $\eps$-neighbourhood of $0$. Set $\phi^\eps_t(x) = \psi(t) \eta^\eps(x)$. Applying the weak formulation of the continuity equation to $\phi^\eps$, and passing to the limit $\eps \downarrow 0$, we obtain
\begin{align*}
	\int_0^1 \psi'(t) \mu_t(\hA_n) \dd t
	+ \int_0^1 \psi(t) 
		\Big(
		 j_t\big(\tfrac{n+1}{N}\big) - j_t\big(\tfrac{n}{N}\big)
		\Big)
	   \dd t 
	= \psi(1) \mu_1(\hA_n) -  \psi(0) \mu_0(\hA_n) \ .
\end{align*}
Multiplying this identity by $\alpha_k^\star$, and using the fact that 
\begin{align*}
	\alpha_k^\star \big( 
		j_t \big(\tfrac{n+1}{N}\big) - 
		j_t \big(\tfrac{n}{N}\big)
		\big)
		 = 
		 J_t(n;k,k+1) - J_t(n;k-1,k) \ ,
\end{align*}
we obtain
\begin{align*}
	& \int_0^1 \psi'(t) m_t(n;k) \dd t
	+ \int_0^1 \psi(t) 
		\Big(
		  J_t(n;k,k+1) - J_t(n;k-1,k)
		\Big)
	   \dd t 
	\\& \qquad  = \psi(1) m_1(n;k) -  \psi(0) m_0(n;k) \ ,
\end{align*}
which is the distributional form of the discrete continuity equation \eqref{eq:cont-eq}.
\end{proof}

\begin{lemma}[Consistency] \label{lem:almost-identity}
For all $\mu \in \cP(\cS^1)$ we have
\begin{align*}
	\bW(\mu, \iota_N P_N^\star \mu) \leq \frac{1}{N} \ .
\end{align*}
\end{lemma}

\begin{proof}
This readily follows from the definitions; see \cite[Lemma 3.2]{Gladbach-Kopfer-Maas:2018} for a similar result.
\end{proof}

The following proposition is the key result of this section. It proves the required upper bound for the discrete energy under suitable regularity conditions.

For $\delta > 0$, it will be useful to write
\begin{align*}
\cP_\delta(\cS^1) := \bigg\{ \mu = \rho \dd x \in \cP(\cS^1)
\suchthat
 	\rho \geq \delta > 0 , \quad
	\Lip(\rho) \leq \frac{1}{\delta} \bigg\} \ .
\end{align*}

\begin{proposition}[Discrete energy upper bound]\label{prop:action-upper-bound}
Let $\delta > 0$. There exists $C < \infty$ and $\bar N \in \N$ (depending on $\delta$), such that for any $N \geq \bar N$, all $\mu \in \cP_\delta(\cS^1)$, and all vector fields $j : \cS^1 \to \R$ with $\|j\|_{L^\infty} + \Lip(j) \leq \delta^{-1}$, we have
\begin{align*}
	\cA_N(P_N^\star \mu, P_N^\star j)
		\leq c^\star(\theta, \cT) \bA(\mu,j) + \frac{C}{N} \ .
\end{align*}
\end{proposition}

\begin{proof}
Write $m = P_N^\star \mu$ and $J = P_N^\star j$ for brevity, and set $\bar\rho(n) := N \mu(\hA_n)$.
Recall that
\begin{align*}
\cA_N(m, J) = \frac{1}{N} \sum_{n=0}^{N-1} \cA_N^n(m, J)
\end{align*}
where
\begin{align*}
	\cA_N^n(m,J)
	& = \sum_{k=0}^{K-1}
		d_{k,k+1} f_{k,k+1}\bigg(N \frac{m(n;k)}{\pi_k}, N \frac{m(n;k+1)}{\pi_{k+1}}, J(n;k,k+1) \bigg)
	\\& = \frac{1}{\bar\rho(n)}\sum_{k=0}^{K-1} d_{k,k+1}
		\frac{(J(n;k,k+1))^2}{ \theta_{k,k+1} \big( \frac{\alpha(n;k)}{\pi_k}, \frac{\alpha(n;k+1)}{\pi_{k+1}} \big)} \ ,
\end{align*}
with $\alpha(n;k) := \frac{\emm(n;k)}{\mu(\hA_n)}$ for $k = 0, \ldots, K$.
Note that $\alpha(n;k) = \alpha_k^\star$ for $k = 0, \ldots, K-1$, but
\begin{align*}
	\alpha(n;K)
		= \frac{m(n+1;0)}{\mu(\hA_n)}
		= \alpha_0^\star \frac{\mu(\hA_{n+1})}{\mu(\hA_n)} \ ,
\end{align*}
which is not necessarily equal to $\alpha_K^\star = \alpha_0^\star$.
Therefore, $\{ \alpha(n;k) \}_{k = 0}^{K}$ is not necessarily an admissible competitor in \eqref{eq: hom-constant for K-periodic case}.
Write $\dd \mu(x) = \rho(x) \dd x$.
We claim that the following estimates hold for sufficiently large $N$, with $C < \infty$ depending only on $\theta$ and $\cT$:
\begin{align}
\label{eq:J-bound}
	| J(n;k,k+1) - j(x) |
	& \leq \frac{\Lip(j)}{N} \ , \\
\label{eq:theta-bound}
	\Bigg| c^\star(\theta, \cT) - \sum_{k=0}^{K-1} \frac{d_{k,k+1}}
		{ \theta_{k,k+1} \Big( \frac{\alpha(n;k)}{\pi_k}, \frac{\alpha(n;k+1)}{\pi_{k+1}} \Big)} \Bigg| &  \leq 	C
					\frac{\Lip(\rho)}{\inf \rho} \frac{1}{N}\ ,
\end{align}
the first one being valid for any $x \in \big[\frac{n}{N}, \frac{n+1}{N}\big]$ and $k = 0, \ldots, K-1$. Indeed, writing $\lambda_k = \sum_{\ell={k+1}}^{K-1} \alpha_\ell^\star$, we obtain
\[
| J(n;k,k+1) - j(x) |
	\leq \lambda_k | j(\tfrac{n}{N}) - j(x) |
				+ (1-\lambda_k) | j(\tfrac{n+1}{N}) - j(x)|
	\leq \frac{\Lip(j)}{N} \ ,
\]
which proves \eqref{eq:J-bound}. 
Furthermore, 
\begin{align*}
	E & := \left| c^\star(\theta, \cT)
		- \sum_{k=0}^{K-1} \frac{d_{k,k+1}}{
			\theta_{k,k+1} \Big(
			 \frac{\alpha(n;k)}{\pi_k},
			 \frac{\alpha(n;k+1)}{\pi_{k+1}} \Big)} \right|
	\\&\phantom{:}=
		d_{K-1,K} \left|
		\frac{1}{ \theta_{K-1, K} \Big(
			\frac{\alpha_{K-1}^\star}{\pi_{K-1}},
			\frac{\alpha_K^\star}{\pi_K} \Big)}
			- \frac{1}{\theta_{K-1, K} \Big(
				\frac{\alpha_{K-1}^\star}{\pi_{K-1}},
				\frac{\alpha(n;K)}{\pi_K} \Big)}
				\right| \ .
\end{align*}
Note that
\begin{align*}
	|  \alpha_K^\star - \alpha(n;K)|
		= \alpha_K^\star \bigg| 1 - \frac{\mu(\hA_{n+1})}{\mu(\hA_n)} \bigg|
	  = \alpha_K^\star \frac{| \mu(\hA_n) - \mu(\hA_{n+1}) |}{\mu(\hA_n)}
	 \leq \frac{\alpha_K^\star}{N}
				\frac{\Lip(\rho)}{\inf \rho} \ .
\end{align*}
If $\alpha_K^\star = 0$, we infer that $E = 0$, in which case the claim is proved. 
If $\alpha_K^\star > 0$, we observe that the latter inequality yields
\begin{equation} \label{eq: upp bound-lower bound density}
	\alpha(n;K) \geq 
\frac{\alpha_K^\star}{2}
\end{equation}
for $N$ sufficiently large (depending on $\delta$).
Since $\theta_{K-1, K}$ is concave, we have for any $a \geq 0$ and $0 < b \leq y_1 < y_2$,
\begin{align*}
	  \frac{\theta_{K-1, K} 
		\big(a, y_2\big)
	- \theta_{K-1, K} 
		\big(a, y_1\big)}{y_2 - y_1}	
		\leq
	  \frac{\theta_{K-1, K} 
		\big(a, b\big)
	- \theta_{K-1, K} 
		\big(a, 0\big)}{b} \ ,
\end{align*}
thus $\theta_{K-1, K}(a, \cdot)$ is Lipschitz on $[b, \infty)$.
Let $L < \infty$ denote the Lipschitz constant of $\theta_{K-1, K}\big(\tfrac{\alpha_{K-1}^\star}{\pi_{K-1}}, \cdot \big)$ on $\big[ \frac{\alpha_K^\star}{2}, \infty \big)$.
For $N$ sufficiently large we obtain
\begin{align*}
E & \leq 
		 \frac{1}{\Big(\theta_{K-1, K} \Big(
				\frac{\alpha_{K-1}^\star}{\pi_{K-1}},
				\frac{\alpha_K^\star}{2\pi_K} \Big)\Big)^2}
			\Big| 
					{ \theta_{K-1, K} \Big(
			\tfrac{\alpha_{K-1}^\star}{\pi_{K-1}},
			\tfrac{\alpha_K^\star}{\pi_K} \Big)}
			- {\theta_{K-1, K} \Big(
				\tfrac{\alpha_{K-1}^\star}{\pi_{K-1}},
				\tfrac{\alpha(n;K)}{\pi_K} \Big)}
			 \Big|								 
	\\ & \leq
		 \frac{L}{\Big(\theta_{K-1, K} \Big(
				\tfrac{\alpha_{K-1}^\star}{\pi_{K-1}},
				\tfrac{\alpha_K^\star}{2\pi_K} \Big)\Big)^2}
				\frac{\alpha_K^*}{\pi_K}
				\frac{\Lip(\rho)}{\inf \rho}
				\frac{1}{N}
\end{align*}
which yields our claim \eqref{eq:theta-bound}.

Taking into account that $\bar\rho(n) \geq \delta$ and $\| j \|_\infty \leq \delta^{-1}$, 
it follows from \eqref{eq:J-bound} and a twofold application of \eqref{eq:theta-bound} that
\begin{align*}
	\bigg| \cA_N^n(m, J) - c^\star(\theta, \cT) \frac{j^2\big(\frac{n}{N}\big)}{\bar\rho(n)}
 \bigg|
	& \leq
			 \frac{j^2\big(\tfrac{n}{N}\big)}{\bar\rho(n)}
		\Bigg| c^\star(\theta, \cT)
			- \sum_{k=0}^{K-1} \frac{d_{k,k+1}}
		{ \theta_{k,k+1} \Big( \frac{\alpha(n;k)}{\pi_k}, \frac{\alpha(n;k+1)}{\pi_{k+1}} \Big)} \Bigg|
	\\&	\qquad	+
 \frac{1}{\bar\rho(n)} \sum_{k=0}^{K-1} d_{k,k+1}
		\frac{\big|J^2(n;k,k+1)
			- j^2\big(\frac{n}{N}\big) \big|}{ \theta_{k,k+1}
			\Big( \frac{\alpha(n;k)}{\pi_k},
					    \frac{\alpha(n;k+1)}{\pi_{k+1}} \Big)}
	\\& \leq
  		\frac{C}{N} + \frac{C}{N}
		\sum_k \frac{d_{k,k+1}}{ \theta_{k,k+1}
			\Big( \frac{\alpha(n;k)}{\pi_k},
					    \frac{\alpha(n;k+1)}{\pi_{k+1}} \Big)}	
	\\& \leq \frac{C}{N} \ ,
\end{align*}
where $C < \infty$ depends on $\cT$, $\theta$, and $\delta$.
Consequently,
\begin{align*}
	\cA_N(m, J)
	\leq \frac{  c^\star(\theta,\cT) }{N} \sum_{n=0}^{N-1}
		\frac{j^2\big(\frac{n}{N}\big)}{\bar\rho(n)}
		 	+ \frac{C}{N} \ .
\end{align*}
By the arithmetic-harmonic mean inequality,
\begin{align*}
	\frac{\big|j\big(\frac{n}{N}\big)\big|^2}{\bar\rho(n)}
	\leq N\big|j\big(\tfrac{n}{N}\big)\big|^2 \int_{\hA_n} \frac{1}{\rho(x)} \dd x
	\leq N \int_{\hA_n} \frac{|j(x)|^2}{\rho(x)} \dd x
	   + \frac{C}{N}  \ .
\end{align*}
We infer that
\begin{align*}
	\cA_N(m, J) \leq c^\star(\theta, \cT) \bA(\mu, j) + \frac{C}{N} \ ,
\end{align*}
which completes the proof.
\end{proof}

The previous result shows that the sought upper bound can be achieved once we assume some regularity of the solution of the continuity equation. Therefore in order to conclude the proof of Theorem \ref{thm:main} we seek once again for a regularization procedure.

The following result collects some well-known properties of the heat semigroup $(H_s)_{s \geq 0}$ on $\cP(\cS^1)$.

\begin{lemma}[Regularisation by heat flow]\label{lem:heat-flow-reg}
Let $s > 0$. 
There exists a constant $\delta > 0$ such that for any $\mu \in \cP(\cS^1)$ we have $H_s \mu \in \cP_\delta(\cS^1)$. 
Moreover, $\bW(\mu, H_s \mu) \leq \sqrt{2s}$.
\end{lemma}

\begin{proof}
See, e.g., \cite[Proposition 2.9]{Gigli-Maas:2013} for a proof of these well-known facts.
\end{proof}

We continue with a well-known regularisation result. For the convenience of the reader we include a simple proof.

\begin{lemma}[Smooth approximate action minimisers]\label{lem:smooth-action}
Let $\delta > 0$ and let $\eps > 0$. Then there exists $\tilde \delta \in (0, \delta)$, such that the following assertion holds: for any $\mu_0, \mu_1 \in \cP_\delta(\cS^1)$ there exists a curve $(\mu_t, j_t) \in \bCE(\mu_0, \mu_1)$ with $\mu_t \in \cP_{\tilde \delta}(\cS^1)$ and $\|j_t\|_{L^\infty} + \Lip(j_t) \leq \tilde \delta^{-1}$ for any $t \in (0,1)$, such that
\begin{align*}
	\int_0^1 \bA(\mu_t, j_t) \dd t  \leq \bW^2(\mu_0, \mu_1)  + \eps \ .
\end{align*}
\end{lemma}

\begin{proof}
Let $(\mu_t)_{t \in [0,1]}$ be a $\bW$-geodesic connecting $\mu_0$ and $\mu_1$, and let $(j_t)_{t \in [0,1]}$ be a vector field such that
\begin{align*}
	\int_0^1 \bA(\mu_t, j_t) \dd t = \bW^2(\mu_0, \mu_1) \ .
\end{align*}
The idea of the proof is to regularise $\mu_0$ and $\mu_1$ by applying the heat flow for a short time $s > 0$, and then to connect the regularised measures $H_s \mu_0$ and $H_s \mu_1$ using the natural candidate $(H_s \mu_t)_{t \in [0,1]}$.

Firstly, for $i = 0, 1$ and $s > 0$, set $\gamma_t^{i,s} = H_{st} \mu_i$ for $t \in [0,1]$, and let $\rho_t^{i,s}$ be the density of $\gamma_t^{i,s}$ with respect to the Haar measure. 
Then: $\partial_t \gamma_t^{i,s} = s \partial_x^2 \gamma_t^{i,s}$, thus the continuity equation $\partial_t \rho_t^{i,s} + \dive k_t^{i,s} = 0$ holds with $k_t^{i,s} = - s \partial_x \rho_t^{i,s}$.
Using the contractivity of the Fisher information under the heat flow, and the fact that $\mu_i \in \cP_\delta(\cS^1)$, we obtain
\begin{align*}
	\int_0^1 \bA(\gamma_t^{i,s}, k_t^{i,s}) \dd t
	& = 	s^2 \int_0^1  \int_{\cS^1} 
				\frac{|\partial_x \rho_t^{i,s}(x)|^2}{\rho_t^s(x)} 
			\dd x \dd t
   \leq 	s^2 \int_{\cS^1} 
				\frac{|\partial_x \rho_0^{i,s}(x)|^2}{\rho_0^s(x)} 
			\dd x
	  \leq \frac{s^2}{\delta^3} \ .
\end{align*}

Secondly, for any $s > 0$, we note that $(H_s \mu_t, H_s j_t)_{t \in [0,1]}$ solves the continuity equation, and, by the joint convexity of $\bA$ and the fact that $H_s$ is given by a convolution kernel, 
\begin{align*}
	\bA(H_s \mu_t, H_s j_t) \leq \bA(\mu_t, j_t) \ . 
\end{align*}

Fix $\tau \in (0,\frac{1}{4})$, and consider now the curve $(\tilde\mu_t, \tilde j_t)_{t \in [0,1]} \in \bCE(\mu_0, \mu_1)$ defined by 
\begin{align*}
\tilde\mu_t := \begin{cases}
H_{t s/ \tau} \mu_0 &  \\
H_{s} \mu_{(t- \tau)/(1- 2 \tau)} &  \\
H_{(1-t) s/ \tau} \mu_1 &  \\
\end{cases} \ ,
\qquad 
\tilde j_t := \begin{cases}
	\frac{1}{\tau} k_{t/\tau}^{0,s}  
		& t \in ( 0, \tau) \\
	\frac{1}{1 - 2\tau} H_s j_{(t- \tau)/(1- 2 \tau)} 
		& t \in ( \tau, 1 - \tau ) \\
	-\frac{1}{\tau} k_{1-t/\tau}^{1,s} 
		& t \in ( 1- \tau, 1 )
\end{cases} \ ,
\end{align*}
It follows from the bounds above, using the fact that $\frac{1}{1-2\tau}\leq 1 + 4 \tau$ and $\bW^2 \leq \frac14$, that
\begin{align*}
	\int_0^1 \bA(\tilde\mu_t, \tilde j_t) \dd t
	& = \int_0^1 \frac{\bA(\gamma_t^{0,s}, k_t^{0,s}) }{\tau}  
	+ \frac{\bA(H_s \mu_t, H_s j_t)  }{1 - 2\tau} 
	+ \frac{\bA(\gamma_t^{1,s}, k_t^{1,s})}{\tau}  \dd t
	\\&\leq  \frac{s^2}{\delta^3\tau} + \frac{\bW^2(\mu_0, \mu_1)}{1 - 2\tau}  + \frac{s^2}{\delta^3\tau}
	\\&\leq  \frac{s^2}{\delta^3\tau} + (\bW^2(\mu_0, \mu_1)  + \tau) + \frac{s^2}{\delta^3\tau} \ .
\end{align*}
Let $\eps > 0$, and choose $\tau = \eps/2$, and $s^2 = \delta^3 \tau \eps/4$. Then: $\int_0^1 \bA(\tilde\mu_t, \tilde j_t) \dd t \leq \bW^2(\mu_0, \mu_1) + \eps$.

Moreover, by Lemma \ref{lem:heat-flow-reg}, $\tilde \mu_t$ belongs to $\cP_{\tilde \delta}(\cS^1)$ for some $\tilde\delta > 0$ depending on $\delta$ and $s$.
Furthermore, 
\begin{align*}
	\| H_s j_t \|_{L^\infty(\cS^1)}  
		+  \| \partial_x H_s j_t \|_{L^\infty(\cS^1)} 
	\leq C(s) \| j_t\|_{L^1(\cS^1)} 
	\leq C(s) \sqrt{ \bA(\mu_t, j_t) }
		= C(s) \bW(\mu_0, \mu_1)  \ ,
\end{align*}
where the last inequality follows from the Cauchy-Schwarz inequality.
\end{proof}

We are now ready to prove the upper bound in Theorem \ref{thm:main}.

\begin{theorem}[Upper bound for $\cW_N$] 	\label{thm:upperbound_quantitative}
For any mesh $\cT$ and any family of admissible means $\{\theta_{k,k+1}\}_k$ we have 
\begin{align*}
		\limsup_{N \to \infty} \cW_N^2( P_N \mu_0, P_N \mu_1)
		\leq	
		c^\star(\theta,\cT) \bW^2(\mu_0, \mu_1)
		 \ ,
\end{align*}
uniformly for all $\mu_0, \mu_1 \in \cP(\cS^1)$. More precisely,  
for any $\eps > 0$ there exists $\bar N \in \N$ such that for any
$N \geq \bar N$ and $\mu_0, \mu_1 \in \cP(\cS^1)$, we have
\begin{equation} \label{eq:upperbound_quantitative}
		\cW_N^2(P_N \mu_0, P_N \mu_1)
			\leq 
			c^\star(\theta,\cT) \bW^2( \mu_0, \mu_1) 
		 	+ \eps \ .
\end{equation}
\end{theorem}

\begin{proof}
Let $\mu_0, \mu_1 \in \cP(\cS^1)$ and $\eps \in (0,1]$.
By Lemma \ref{lem:heat-flow-reg} there exist $s \geq 0$ and $\delta > 0$ such that $\tilde \mu_i := H_s \mu_i$ belongs to $\cP_\delta(\cS^1)$, and
\begin{align*}
	\bW(\mu_i, \tilde \mu_i) \leq \eps \quad \text{for $i = 0,1$} \ .
\end{align*}
Using that $\bW \leq \frac12$, it follows that
\begin{align}\label{eq:upper-1}
	\bW^2(\tilde \mu_0, \tilde \mu_1)
	\leq 
	\bW^2(\mu_0, \mu_1) + 2 \eps \ .
\end{align}
Lemma \ref{lem:smooth-action} yields $\tilde \delta \in (0,\delta)$ and a curve
 $(\tilde \mu_t, \tilde j_t)_t \in \bCE_{\tilde \delta}(\tilde \mu_0, \tilde \mu_1)$ such that 
  $\tilde \mu_t \in \cP_{\tilde \delta}(\cS^1)$ and $\|\tilde j_t\|_{L^\infty} + \Lip(\tilde j_t) \leq \tilde \delta^{-1}$ for any $t \in (0,1)$, and
\begin{align}\label{eq:upper-2}
	\int_0^1 \bA(\tilde \mu_t, \tilde j_t) \dd t 
		\leq
	\bW^2(\tilde \mu_0, \tilde \mu_1) +  \eps \ .
\end{align}
Set $\tilde m_t^N := P_N^\star \tilde \mu_t$ and $\tilde J_t^N = P_N^\star \tilde j_t$.
By Proposition \ref{prop:action-upper-bound} there exist $\bar N \in \N$ and $C_1 < \infty$ depending on $\eps$ (through $\tilde \delta$) such that for $N \geq \bar N$,
\begin{align} \label{eq:upper-3}
	\cW_N^2(\tilde m_0^N, \tilde m_1^N)
	\leq \int_0^1 \cA(\tilde m_t^N, \tilde J_t^N) \dd t
	\leq c^\star(\theta, r) \int_0^1 \bA(\tilde \mu_t, \tilde j_t) \dd t
		+ \frac{C_1}{N} \ .
\end{align}
Set $\emm_i^N := P_N^\star \mu_i$ for $i = 0,1$. 
By Proposition \ref{prop:rough-upper-bound}, Lemma \ref{lem:almost-identity}, and Lemma \ref{lem:heat-flow-reg}, there exists $C_2 < \infty$ depending only on $\cT$ (possibly varying from line to line) such that
\begin{align*}
	\cW_N(\emm_i^N, \tilde \emm_i^N)
	& = \cW_N(P_N^\star \mu_i, P_N^\star H_s \mu_i)
	\\& \leq C_2 \Big( \bW(\iota_N P_N^\star \mu_i, \iota_N P_N^\star H_s \mu_i) + \frac{1}{N} \Big)
	\\& \leq C_2 \Big( \bW(\mu_i, H_s \mu_i) + \frac{1}{N} \Big)
	\\& \leq C_2 \Big( \sqrt{s} + \frac{1}{N} \Big) \ .
\end{align*}
Thus, the triangle inequality yields
\begin{align*}
	\cW_N(m_0^N, m_1^N)
		\leq \cW_N(\tilde m_0^N, \tilde m_1^N)
			+ C_2 \Big( \sqrt{s} + \frac{1}{N} \Big) \ ,
\end{align*}
and by another application of Proposition \ref{prop:rough-upper-bound}, 
\begin{equation}\begin{aligned}
 \label{eq:upper-4}
	& \cW_N^2(m_0^N, m_1^N)
		 - \cW_N^2(\tilde m_0^N, \tilde m_1^N)
	\\& \qquad  \leq 
 \Big( \cW_N(m_0^N, m_1^N) 
			 	+ \cW_N(\tilde m_0^N, \tilde m_1^N) \Big) C_2 \Big( \sqrt{s} + \frac{1}{N} \Big) 
		\leq   C_2 \Big( \sqrt{s} + \frac{1}{N} \Big)
		 \ .
\end{aligned}\end{equation}

Combining \eqref{eq:upper-1}, \eqref{eq:upper-2}, \eqref{eq:upper-3}, and \eqref{eq:upper-4}, we obtain
\begin{align*}
	\cW_N^2(m_0^N, m_1^N)
	& \leq c^\star(\theta, r) \bW^2(\mu_0, \mu_1)
		+ 3 \eps
		+ \frac{C_1}{N}
		+ C_2 \Big( \sqrt{s} + \frac{1}{N} \Big)		\ .
\end{align*}
Choosing $s$ small enough and $N$ large enough depending on $\eps$, we obtain the result. 
\end{proof}

\section{Proof of the Gromov--Hausdorff convergence}
\label{sec:GH}

We conclude this work with the proof of the Gromov--Hausdorff convergence in Theorem \ref{thm:main}.
First we recall one of the equivalent definitions; cf.  \cite{Burago-Burago-Ivanov:2001} for more details.

\begin{definition}[Gromov--Hausdorff convergence]	
\label{def:GHconv}
A sequence of compact metric spaces $\{\cX_N, \dd _N\}_N$ is said to converge in the sense of Gromov--Hausdorff to a compact metric space $(\cX, \dd)$, if there exist maps $f_N:\cX \rightarrow \cX_N$ with the following properties:
\begin{itemize}
\item $\eps$-isometry: for any $\eps > 0$ there exists $\bar N \in	\N$ such that for any $N \geq \bar N$ and any $x,y \in \cX$, we have:
\begin{align*}
		| \dd_N(f_N(x), f_N(y)) - \dd (x,y) | \leq \eps \ ;
\end{align*}
\item $\eps$-surjectivity: for any $\eps > 0$ there exists $\bar N \in	\N$ such that for any $N \geq \bar N$ and any $z \in \cX_N$ there exists $x \in \cX$ satisfying
\begin{align*}
		\dd_N(f_N(x),z) \leq \eps \ .
\end{align*}
\end{itemize}
\end{definition}

\begin{proof}[Proof of Theorem \ref{thm:main}]
As the desired lower and upper bounds for the distance have been proved in Theorems \ref{thm:lowerbound_quantitative} and \ref{thm:upperbound_quantitative}, it remains to prove the Gromov--Hausdorff convergence. We will show that the conditions above hold with $f_N := P_N$.

Let $\eps > 0$. 
It follows from Theorems \ref{thm:lowerbound_quantitative} and \ref{thm:upperbound_quantitative} that there exists $\bar N \in \N$ such that, for any $N \geq \bar N$ and $\mu_0, \mu_1 \in \cP(\cS^1)$,
\begin{align*}
	\big| \cW_N( P_N \mu_0, P_N \mu_1 ) - \sqrt{c^\star(\theta,\cT)} \bW (\mu_0,\mu_1)   \big| 
	\leq \eps \ .
\end{align*}
This shows that the map $P_N$ is $\eps$-isometric.

The $\eps$-surjectivity of $P_N$ holds trivially, since it is even surjective.
\end{proof}

\small
\subsection*{Acknowledgements}

J.M. gratefully acknowledges support by the European Research Council (ERC) under the European Union's Horizon 2020 research and innovation programme (grant agreement No 716117).
J.M and L.P. also acknowledge support from the Austrian Science Fund (FWF), grants No F65 and W1245. 
E.K. gratefully acknowledges support by the German Research Foundation through the Hausdorff Center for Mathematics and the Collaborative Research Center 1060.
P.G. is partially funded by the Deutsche Forschungsgemeinschaft (DFG, German Research Foundation) -- 350398276.

\normalsize

\bibliography{literature}{}
\bibliographystyle{jan}
\end{document}